\documentclass[reqno]{amsart}

\usepackage[letterpaper,top=2cm,bottom=2cm,left=3cm,right=3cm,marginparwidth=2.0cm]{geometry}

\usepackage[linesnumbered,ruled,vlined]{algorithm2e}

\usepackage{amsmath,amsthm,amssymb,mathtools, mathrsfs,booktabs} 
\usepackage[english]{babel}
\usepackage[utf8]{inputenc}
\usepackage{microtype} 
\usepackage{csquotes}
\usepackage{graphicx}
\usepackage{ytableau}
\usepackage{tikz}
\usepackage{verbatim}
\usetikzlibrary{decorations.pathreplacing}
\usepackage{todonotes}
\presetkeys{todonotes}{inline}{}
\usepackage{arydshln}

\usepackage[colorlinks=true, allcolors=blue,pdftex]{hyperref}

\usepackage{xcolor}
\newcommand{\Lilan}{\color{black}}

\theoremstyle{definition}
\newtheorem{example}{Example}
\newtheorem{definition}{Definition}
\newtheorem{remark}{Remark}

\theoremstyle{plain}
\newtheorem{theorem}{Theorem}
\newtheorem{lemma}{Lemma}
\newtheorem{proposition}{Proposition}
\newtheorem{corollary}{Corollary}

\definecolor{lightblue}{RGB}{173, 216, 230}
\definecolor{skyblue}{RGB}{135, 206, 255}
\definecolor{pastelblue}{RGB}{176, 224, 230}
\definecolor{iceblue}{RGB}{200, 230, 250}

\definecolor{mintgreen}{RGB}{144, 238, 144}     
\definecolor{pastelgreen}{RGB}{152, 251, 152}   
\definecolor{softgreen}{RGB}{193, 255, 193} 
\definecolor{lemonyellow}{RGB}{255, 250, 205}
\definecolor{peachpink}{RGB}{255, 218, 185}
\definecolor{lavender}{RGB}{230, 230, 250}
\definecolor{aquablue}{RGB}{175, 238, 238}

\newcommand{\defin}[1]{%
\relax\ifmmode%
\textcolor{blue}{#1}%
\else\textcolor{blue}{\emph{#1}}%
\fi%
}

\newcommand{\NN}{\mathbb{N}}
\newcommand{\QQ}{\mathbb{Q}}
\newcommand{\ZZ}{\mathbb{Z}}

\newcommand{\Rset}{\mathcal R_{k,N}}
\newcommand{\uvec}{\mathbf{u}}

\newcommand{\wt}{\operatorname{wt}}
\newcommand{\tr}{\operatorname{tr}}

\DeclareMathOperator{\rowMap}{\mathtt{rowDiv}}

\newcommand{\monomial}{\mathrm{m}}
\newcommand{\powersum}{\mathrm{p}}
\newcommand{\completeH}{\mathrm{h}}
\newcommand{\elementaryE}{\mathrm{e}}

\makeatletter

\ExplSyntaxOn

\seq_new:N \l__rht_rows_seq
\int_new:N \l__rht_rowcount_int
\bool_new:N \l__rht_french_bool
\bool_new:N \l__rht_dot_bool
\tl_new:N \l__rht_scale_tl
\tl_new:N \l__rht_currow_tl
\tl_new:N \l__rht_curdigit_tl
\tl_new:N \l__rht_neighbor_tl
\int_new:N \l__rht_curcols_int

\NewDocumentCommand{\rimHookTableau}{ O{} m }
{
  \tl_set:Nn \l__rht_scale_tl { 1.0 }
  \bool_set_false:N \l__rht_french_bool
  \bool_set_false:N \l__rht_dot_bool
  \keys_set:nn { rht } { #1 }
  \__rht_parse_and_draw:n { #2 }
}

\keys_define:nn { rht }
{
  scale .tl_set:N = \l__rht_scale_tl,
  French .bool_set:N = \l__rht_french_bool,
  Dot .bool_set:N = \l__rht_dot_bool,
}

\cs_new_protected:Nn \__rht_parse_and_draw:n
{
  \seq_set_split:Nnn \l__rht_rows_seq { , } { #1 }
  \int_set:Nn \l__rht_rowcount_int { \seq_count:N \l__rht_rows_seq }
  \__rht_store_cells:
  \__rht_build_commands:
  \__rht_draw:
}

\cs_new_protected:Nn \__rht_store_cells:
{
  \int_step_inline:nn { \l__rht_rowcount_int }
  {
    \tl_set:Nx \l__rht_currow_tl { \seq_item:Nn \l__rht_rows_seq { ##1 } }
    \int_set:Nn \l__rht_curcols_int { \tl_count:N \l__rht_currow_tl }
    \tl_gset:cx { g__rht_rowlen_##1 } { \int_use:N \l__rht_curcols_int }
    \int_step_inline:nn { \l__rht_curcols_int }
    {
      \tl_gset:cx { g__rht_cell_##1_####1 } 
        { \tl_item:Nn \l__rht_currow_tl { ####1 } }
    }
  }
}

\cs_new:Nn \__rht_get_cell:nn
{
  \tl_if_exist:cTF { g__rht_cell_#1_#2 }
    { \tl_use:c { g__rht_cell_#1_#2 } }
    { }
}

\cs_new:Nn \__rht_get_rowlen:n
{
  \tl_if_exist:cTF { g__rht_rowlen_#1 }
    { \tl_use:c { g__rht_rowlen_#1 } }
    { 0 }
}

\cs_new:Nn \__rht_ycoord:n
{
  \bool_if:NTF \l__rht_french_bool
    { \int_eval:n { \l__rht_rowcount_int - #1 + 1 } }
    { #1 }
}

\cs_new_protected:Nn \__rht_build_commands:
{
  \gdef\rhtcmds{}
  
  \int_step_inline:nn { \l__rht_rowcount_int }
  {
    \int_step_inline:nn { \__rht_get_rowlen:n { ##1 } }
    { \__rht_add_cell_cmd:nn { ##1 } { ####1 } }
  }
  
  \int_step_inline:nn { \l__rht_rowcount_int }
  {
    \int_step_inline:nn { \__rht_get_rowlen:n { ##1 } }
    { \__rht_add_connection_cmd:nn { ##1 } { ####1 } }
  }
  
  \bool_if:NT \l__rht_dot_bool
  {
    \int_step_inline:nn { \l__rht_rowcount_int }
    {
      \int_step_inline:nn { \__rht_get_rowlen:n { ##1 } }
      { \__rht_add_dot_cmd:nn { ##1 } { ####1 } }
    }
  }
}

\cs_new_protected:Nn \__rht_add_cell_cmd:nn
{
  \tl_set:Nx \l__rht_curdigit_tl { \__rht_get_cell:nn{#1}{#2} }
  \str_if_eq:VnF \l__rht_curdigit_tl { 0 }
  {
    \edef\rht@tempa{\int_eval:n{#2-1}}
    \edef\rht@tempb{\int_eval:n{-\__rht_ycoord:n{#1}}}
    \exp_args:NNx \g@addto@macro \rhtcmds { \exp_not:N \rhtcell{\rht@tempa}{\rht@tempb} }
  }
}

\cs_new_protected:Nn \__rht_add_connection_cmd:nn
{
  \tl_set:Nx \l__rht_curdigit_tl { \__rht_get_cell:nn{#1}{#2} }
  \tl_if_empty:NF \l__rht_curdigit_tl
  {
    \str_if_eq:VnF \l__rht_curdigit_tl { 0 }  
    {
      \int_compare:nNnT {#2} < { \__rht_get_rowlen:n{#1} }
      {
        \tl_set:Nx \l__rht_neighbor_tl { \__rht_get_cell:nn{#1}{\int_eval:n{#2+1}} }
        \tl_if_eq:NNT \l__rht_curdigit_tl \l__rht_neighbor_tl
        { \__rht_add_hline_cmd:nn{#1}{#2} }
      }

      \bool_if:NTF \l__rht_french_bool
      {
        \int_compare:nNnT {#1} > {1}
        {
          \int_compare:nNnT {#2} < { \int_eval:n{ \__rht_get_rowlen:n{\int_eval:n{#1-1}} + 1 } }
          {
            \tl_set:Nx \l__rht_neighbor_tl { \__rht_get_cell:nn{\int_eval:n{#1-1}}{#2} }
            \tl_if_eq:NNT \l__rht_curdigit_tl \l__rht_neighbor_tl
            { \__rht_add_vline_cmd:nn{#1}{#2} }
          }
        }
      }
      {
        \int_compare:nNnT {#1} < {\l__rht_rowcount_int}
        {
          \int_compare:nNnT {#2} < { \int_eval:n{ \__rht_get_rowlen:n{\int_eval:n{#1+1}} + 1 } }
          {
            \tl_set:Nx \l__rht_neighbor_tl { \__rht_get_cell:nn{\int_eval:n{#1+1}}{#2} }
            \tl_if_eq:NNT \l__rht_curdigit_tl \l__rht_neighbor_tl
            { \__rht_add_vline_cmd:nn{#1}{#2} }
          }
        }
      }
    }
  }
}

\cs_new_protected:Nn \__rht_add_hline_cmd:nn
{
  \edef\rht@tempa{\int_eval:n{#2}}
  \edef\rht@tempb{\int_eval:n{-\__rht_ycoord:n{#1}}}
  \exp_args:NNx \g@addto@macro \rhtcmds { \exp_not:N \rhthline{\rht@tempa}{\rht@tempb} }
}

\cs_new_protected:Nn \__rht_add_vline_cmd:nn
{
  \edef\rht@tempa{\int_eval:n{#2}}
  \edef\rht@tempb{\int_eval:n{-\__rht_ycoord:n{#1}}}
  \bool_if:NTF \l__rht_french_bool
  {
    \edef\rht@tempc{\int_eval:n{-\__rht_ycoord:n{\int_eval:n{#1-1}}}}
  }
  {
    \edef\rht@tempc{\int_eval:n{-\__rht_ycoord:n{\int_eval:n{#1+1}}}}
  }
  \exp_args:NNx \g@addto@macro \rhtcmds { \exp_not:N \rhtvline{\rht@tempa}{\rht@tempb}{\rht@tempc} }
}

\cs_new_protected:Nn \__rht_add_dot_cmd:nn
{
  \tl_set:Nx \l__rht_curdigit_tl { \__rht_get_cell:nn{#1}{#2} }
  \str_if_eq:VnF \l__rht_curdigit_tl { 0 }
  {
    \edef\rht@tempa{\int_eval:n{#2}}
    \edef\rht@tempb{\int_eval:n{-\__rht_ycoord:n{#1}}}
    \exp_args:NNx \g@addto@macro \rhtcmds { \exp_not:N \rhtdot{\rht@tempa}{\rht@tempb} }
  }
}

\cs_new_protected:Nn \__rht_draw:
{
  \edef\rhtscale{\l__rht_scale_tl}
  \rhtdrawtableau
}

\ExplSyntaxOff

\def\rhtcell#1#2{\draw (#1,#2) rectangle +(1,1);}
\def\rhthline#1#2{\draw[line width=1.5pt] (#1-0.5,#2+0.5) -- (#1+0.5,#2+0.5);}
\def\rhtvline#1#2#3{\draw[line width=1.5pt] (#1-0.5,#2+0.5) -- (#1-0.5,#3+0.5);}
\def\rhtdot#1#2{\fill (#1-0.5,#2+0.5) circle (2pt);}

\def\rhtdrawtableau{%
  \begin{tikzpicture}[scale=\rhtscale,baseline=(current bounding box.center)]
    \rhtcmds
  \end{tikzpicture}%
}

\makeatother

\title{The $\elementaryE$-positivity of $\rowMap_k(\elementaryE_\mu)$}

\author{Per Alexandersson}
\address{Department of Mathematics, Stockholm University, SE-106 91 Stockholm, Sweden}
\email{per.w.alexandersson@gmail.com}

\author{Lilan Dai}
\address{Center for Combinatorics, LPMC, 
	Nankai University, Tianjin 300071, P. R. China}
\email{daililan@mail.nankai.edu.cn}

\begin{document}

\begin{abstract}
We prove that the partition division operator $\rowMap_k$ preserves $\elementaryE$-positivity, resolving a conjecture
of Alexandersson and Dai. For $\mu\vdash kn$ and $\nu\vdash n$, we express the coefficient
of $\elementaryE_\nu$ in $\rowMap_k(\elementaryE_\mu)$ as a
normalized sum of positive rational weights over closed walks on a finite directed graph of $k$-restricted partitions. The proof combines the Pieri rule and a root of unity
factorization for Macdonald polynomials.

As a corollary, we recover the $\completeH$-positivity of the symmetric functions $A_{n,k}$ of Amdeberhan, Shareshian, and Stanley.
\medskip 

\medskip 
		
\noindent\textbf{AMS Classification 2020:} 05E05; 33D52.
\end{abstract}

\maketitle

\section{Introduction}\label{sec:intro}
Let $\Lambda$ be the ring of symmetric functions over $\QQ$, and let $\Lambda_{\ZZ}$ be the subring with integral monomial coefficients. We use the monomial basis $\{\monomial_\alpha\}$ and the elementary basis $\{\elementaryE_\alpha\}$; our conventions are recalled in Section~\ref{sec:prelim}. Following Alexandersson and Dai~\cite{AlexanderssonDai2026}, we study the linear map $\rowMap_k:\Lambda\to\Lambda$ by
\begin{equation}\label{eq:rowdiv}
 \rowMap_k(\monomial_\alpha)=
 \begin{cases}
  \monomial_{\alpha/k},&k\mid\alpha_i\text{ for every }i,\\
  0,&\text{otherwise}.
 \end{cases}
\end{equation}
Here $\alpha/k$ is obtained by dividing each part by $k$.

For $\mu\vdash kn$ and $\nu\vdash n$, let
$c^{(k)}_{\mu\nu}$ denote the coefficient of $\elementaryE_\nu$
in the elementary expansion of $\rowMap_k(\elementaryE_\mu)$.
These coefficients are integers, since $\rowMap_k$ preserves
$\Lambda_{\ZZ}$ and the $\elementaryE_\nu$ form a $\ZZ$-basis. Alexandersson and Dai conjectured that
$c^{(k)}_{\mu\nu}\ge0$ for all
$\mu\vdash kn$ and $\nu\vdash n$
\cite[Conjecture~1]{AlexanderssonDai2026}.
We prove this conjecture. 

Since each $\elementaryE_\lambda$ is Schur-positive,
$\elementaryE$-positivity implies Schur positivity, while in many naturally
occurring examples the coefficients in the elementary expansion encode
more refined combinatorial or representation-theoretic information.
A prominent source of such questions is the theory of chromatic
symmetric functions initiated by Stanley\cite{Stanley1995}, where the
Stanley--Stembridge conjecture placed $\elementaryE$-positivity at the center of
the subject\cite{StanleyStembridge1993}. Guay-Paquet~\cite{GuayPaquet2013} reduced the conjecture to the case of
unit interval graphs, which was subsequently proved by
Hikita~\cite{Hikita2026}. Closely related positivity
phenomena also arise in the theory of LLT and Macdonald polynomials;
in particular, the $\elementaryE$-positivity and elementary expansions of
vertical-strip LLT polynomials have been studied from several
combinatorial and algebraic viewpoints
\cite{DAdderio2020,AlexanderssonSulzgruber2020x,GarsiaHaglundQiuRomero2019}. For a positive integer $k$, the map $\rowMap_k$ retains, in the
monomial basis, only those terms whose parts are all divisible by $k$
and divides the surviving parts by $k$. Although this definition is
entirely monomial in nature, its action on products of elementary
symmetric functions exhibits a strong positivity phenomenon. In related work, Siegl~\cite[Theorem~4.6]{Siegl2026} proves Schur positivity for certain weighted sums of chromatic symmetric functions. He also proposes an $\elementaryE$-positivity conjecture~\cite[Conjecture~4.8]{Siegl2026} that implies the Alexandersson--Dai conjecture proved here.

The proof is based on a finite weighted graph. Fix $N\ge n$. Its vertices are the $k$-restricted partitions with at most $N$ parts. Starting from a vertex $\lambda$, choose a vertical strip $\eta/\lambda$. The partition $\eta$ has a unique decomposition $ \eta=\tau+k\xi$, where $\tau$ is $k$-restricted. We regard this operation as an edge from $\lambda$ to $\tau$. The conjugate partition $\xi'$ is recorded as the label of the edge, and the edge receives an explicit positive rational weight, a product of local factors indexed by pairs of rows of $\lambda$. A walk is obtained by repeating this operation. The main formula states that $c^{(k)}_{\mu\nu}$ is $k^{-N}$ times the total weight of the closed walks whose successive strip sizes are $\mu_1,\ldots,\mu_{\ell(\mu)}$ and whose combined label is $\nu$.

The organization of this paper is as follows.
In Section~\ref{sec:prelim}, we review the necessary preliminaries
on symmetric functions and Macdonald polynomials.
In Section~\ref{sec:walks}, we construct a finite directed graph
on $k$-restricted partitions, define its edge labels and positive
rational weights, and introduce the closed walks appearing in
our coefficient formula.
In Section~\ref{sec:macdonald}, we derive a root of unity factorization of Macdonald polynomials in a form adapted to $k$-restricted partitions in
Corollary~\ref{cor:factorization}, which explains the graph construction.
In Section~\ref{sec:evaluation}, we use interpolation and a
root of unity filter to obtain the trace identity in
Proposition~\ref{prop:trace}, relating the transition matrices
to $\rowMap_k(\elementaryE_\mu)$. Finally,
in Section~\ref{sec:proof}, we derive the weighted walk formula
in Theorem~\ref{thm:walk} and prove the $\elementaryE$-positivity
of $\rowMap_k(\elementaryE_\mu)$ in Theorem~\ref{thm:positive}.

{\Lilan
A Lean~4 formalization of Theorem~\ref{thm:positive} and of the walk formula in Theorem~\ref{thm:walk} is available at \url{https://github.com/PerAlexandersson/PartitionDivisionLean}. The proof below follows the formalization; in particular, the Pieri coefficients enter only through their form~\eqref{eq:rowpairs} indexed by pairs of rows.}

\section{Preliminaries}\label{sec:prelim}
\subsection{Partitions and vertical strips}
A \defin{partition} is a weakly decreasing sequence
\[
 \lambda=(\lambda_1,\lambda_2,\ldots)
\]
of nonnegative integers with only finitely many nonzero terms. Its size and length are
\[
 |\lambda|=\sum_i\lambda_i,
 \qquad
 \ell(\lambda)=\#\{i:\lambda_i>0\}.
\]
We write $\lambda\vdash n$ when $|\lambda|=n$, and $\varnothing$ for the partition of zero. Trailing zero parts may be added or omitted as convenient. The \defin{conjugate partition} is
\[
 \lambda'_j=\#\{i:\lambda_i\ge j\}.
\]
Addition and scalar multiplication of partitions are componentwise. Thus $(k^h)=k(1^h)$ is the rectangular partition with $h$ rows of length $k$.

Let $\lambda\subseteq\eta$ be partitions of length at most $N$, padded by zeros to length $N$. We say that $\eta/\lambda$ is a \defin{vertical $r$-strip} if
\[
 \eta=\lambda+\delta,\qquad \delta\in\{0,1\}^N,\qquad \sum_{i=1}^N\delta_i=r.
\]
In the Young diagram, $\eta$ is obtained from $\lambda$ by adding $r$ cells, no two of them in the same row. We call $\delta$ the \defin{strip vector}. If $\delta_i=0$ and $\delta_j=1$ for some $i<j$, then $\lambda_i>\lambda_j$. Indeed, if $\lambda_i=\lambda_j$, then rows $i,i+1,\ldots,j$ of $\lambda$ have the same length, and adding a cell to row $j$ but not to row $i$ violates a partition inequality between them.

\subsection{Symmetric functions}
The \defin{monomial symmetric function} $\monomial_\lambda$ is the sum of the distinct monomials whose nonzero exponents are the parts of $\lambda$. For $r\ge1$, 
\[
 \elementaryE_r=\sum_{i_1<\cdots<i_r}x_{i_1}\cdots x_{i_r},
 \qquad \powersum_r=\sum_i x_i^r,
 \qquad \elementaryE_0=1.
\]
For a partition $\lambda$, let
\[
 \elementaryE_\lambda
=\prod_{j=1}^{\ell(\lambda)}\elementaryE_{\lambda_j},
 \qquad
 \powersum_\lambda
=\prod_{j=1}^{\ell(\lambda)}\powersum_{\lambda_j}.
\]
The families $\{\elementaryE_\lambda\}$ and
$\{\powersum_\lambda\}$ are the elementary and power-sum
bases of $\Lambda$, respectively.
A symmetric function is \defin{$\elementaryE$-positive}
if its coefficients in the elementary basis are nonnegative.

For a field $F$, write $\Lambda_N(F)=F[x_1,\ldots,x_N]^{\mathfrak S_N}$. We use the finite alphabet $X_N=(x_1,\ldots,x_N)$ and write $f[X_N]=f(x_1,\ldots,x_N)$. If $Y$ is an alphabet and $z$ is one additional variable, then $f[Y+z]$ denotes evaluation in their union. We only need the plethystic specialization
\[
 f[p_k](X_N)=f(x_1^k,\ldots,x_N^k).
\]
Define
\begin{equation}\label{eq:Ei} \defin{E_i(X_N)}\coloneqq\elementaryE_i[\powersum_k](X_N),\qquad 1\le i\le N.
\end{equation}
The polynomials $E_1(X_N),\ldots,E_N(X_N)$ are algebraically independent. Indeed, with $y_i=x_i^k$, we have
\[
 \QQ[y_1,\ldots,y_N]^{\mathfrak S_N}
=\QQ[\elementaryE_1(y),\ldots,\elementaryE_N(y)],
\]
and the generators $\elementaryE_1(y),\ldots,\elementaryE_N(y)$ are algebraically independent. Since
\[
 E_i(X_N)=\elementaryE_i(y_1,\ldots,y_N),
\]
the same is true of the $E_i$. 
Since $E_1(X_N),\ldots,E_N(X_N)$ are algebraically independent,
the monomials
\[
 E_\nu(X_N)
 =\prod_{i=1}^N E_i(X_N)^{m_i(\nu)},
 \qquad \nu_1\le N,
\]
are linearly independent over $\QQ$.
We use this fact to compare coefficients in
Section~\ref{sec:proof}.

For a polynomial
\[
 f[X_N]=\sum_{\alpha\in\mathbb Z_{\ge0}^N}
 a_\alpha x_1^{\alpha_1}\cdots x_N^{\alpha_N},
\]
define
\[
 \defin{\Phi_k(f[X_N])}\coloneqq
 \sum_{\substack{\alpha\in\mathbb Z_{\ge0}^N\\k\mid\alpha_i\text{ for every }i}}
 a_\alpha x_1^{\alpha_1}\cdots x_N^{\alpha_N}.
\]
Thus $\Phi_k$ keeps exactly those monomials in which every variable exponent is divisible by $k$.

\begin{lemma}\label{lem:projection}
For every symmetric function $f$ and every $N$,
\begin{equation}\label{eq:projection}
 (\rowMap_k f)[\powersum_k](X_N)=\Phi_k(f[X_N]).
\end{equation}
\end{lemma}
\begin{proof}
Both sides are linear in $f$, so it suffices to take $f=\monomial_\alpha$. By definition, $\monomial_\alpha[X_N]$ is the sum of the distinct monomials obtained by permuting the exponent vector
\[ (\alpha_1,\ldots,\alpha_{\ell(\alpha)},0,\ldots,0).
\]
Hence all exponents in every monomial are divisible by $k$ if and only if every part of $\alpha$ is divisible by $k$. If some part is not divisible by $k$, then both sides of~\eqref{eq:projection} vanish. If $\alpha=k\beta$, then
\[
 \rowMap_k(\monomial_\alpha)=\monomial_\beta,
\]
and
\[
 \monomial_\beta[\powersum_k](X_N)
 =\monomial_\beta(x_1^k,\ldots,x_N^k)
 =\monomial_{k\beta}(X_N)=\monomial_\alpha(X_N).
\]
Thus the two sides agree on every monomial symmetric function and hence on every symmetric function.
\end{proof}

Notice that if $f=\sum_\alpha a_\alpha \monomial_\alpha$ with $a_\alpha\in\ZZ$, then
\[
 \rowMap_k(f)=\sum_{\alpha=k\beta}a_\alpha \monomial_\beta
\]
still has integral monomial coefficients. Since both the monomial and elementary bases are $\ZZ$-bases of $\Lambda_{\ZZ}$, the coefficients $c^{(k)}_{\mu\nu}$ are integers.

\subsection{Macdonald polynomials, Pieri, and the first difference operator}\label{subsec:macdonald-background}
We recall the standard Macdonald-polynomial facts. Here $q$ and $t$ are generic parameters; Section~\ref{sec:macdonald} studies what happens after specializing $q$ to a primitive $k$th root of unity.

For partitions of the same size, write $\alpha\le\lambda$ in dominance order if
\[
 \sum_{i=1}^j\alpha_i\le\sum_{i=1}^j\lambda_i
\]
for every $j$. The \defin{Macdonald symmetric functions $P_\lambda[X;q,t]$} are characterized by
\[
 P_\lambda=\monomial_\lambda+
 \sum_{\alpha<\lambda}u_{\lambda\alpha}(q,t)\monomial_\alpha
\]
and by pairwise orthogonality with respect to the scalar product
\[
 \langle \powersum_\alpha,\powersum_\gamma\rangle_{q,t}
 =\delta_{\alpha\gamma}z_\alpha
   \prod_j\frac{1-q^{\alpha_j}}{1-t^{\alpha_j}},
 \qquad
 z_\alpha=\prod_{i\ge1}i^{m_i(\alpha)}m_i(\alpha)!,
\]
where $m_i(\alpha)$ denotes the multiplicity of the part $i$ in $\alpha$. {\Lilan
We use the following standard properties from
\cite[Chapter VI]{Macdonald1995}. In $N$ variables, $P_\lambda$ vanishes when
$\ell(\lambda)>N$, and the remaining polynomials form a basis.
Also $P_{(1^r)}=\elementaryE_r$ and $P_\varnothing=1$. A rational function is said to be \defin{regular} at a parameter value
if it admits a representation whose denominator is nonzero there. The two standard inputs used below are the vertical Pieri rule and
the eigenvalue identity for the first Macdonald difference operator.}

The \defin{vertical Pieri rule}~\cite[Chapter~VI, (6.24)]{Macdonald1995} states that
\begin{equation}\label{eq:pieri}
 P_\lambda \elementaryE_r=
 \sum_{\eta/\lambda\text{ vertical }r\text{-strip}}
 \psi'_{\eta/\lambda}(q,t)P_\eta .
\end{equation}
We use the coefficient in the form indexed by pairs of rows~\cite[Eq.~(2.12)]{BorodinCorwin2014}: if $\eta=\lambda+\delta$, then
\begin{equation}\label{eq:rowpairs}
 \psi'_{\eta/\lambda}(q,t)=
 \prod_{\substack{1\le i<j\le N\\\delta_i=0,\ \delta_j=1}}
 \frac{(1-q^{\lambda_i-\lambda_j}t^{j-i-1})
       (1-q^{\lambda_i-\lambda_j-1}t^{j-i+1})}
      {(1-q^{\lambda_i-\lambda_j}t^{j-i})
       (1-q^{\lambda_i-\lambda_j-1}t^{j-i})}.
\end{equation}
It agrees with Macdonald's product over cells, by a telescoping product over the rows below each cell. It depends only on the row differences $\lambda_i-\lambda_j$ and the row distances $j-i$, and this is all we use below.

{\Lilan
For the rectangular specialization we also use the first Macdonald difference operator in $N$ variables.  Write
\[
 (T_{a,x_i}f)(x_1,\ldots,x_N)
   =f(x_1,\ldots,a x_i,\ldots,x_N),
 \qquad
 A_i(X_N;t)=\prod_{j\ne i}\frac{t x_i-x_j}{x_i-x_j}.
\]
Then
\begin{equation}
 D_N^1(q,t)=\sum_{i=1}^N A_i(X_N;t)T_{q,x_i}.
 \label{eq:macdonald-first-operator}
\end{equation}
Although the individual $A_i$ are rational functions of $X_N$, this
operator preserves symmetric polynomials and has matrix entries in $\QQ[q,t]$ in the monomial symmetric basis. Indeed, for
$\Delta(X_N)=\prod_{i<j}(x_i-x_j)$ and a symmetric polynomial $f$,
\[
 \Delta D_N^1(q,t)f
   =\sum_{i=1}^N(T_{t,x_i}\Delta)(T_{q,x_i}f)
\]
is alternating and hence divisible by $\Delta$; the quotient is symmetric. Thus the operator can be specialized at $q=\zeta$.

The standard eigenvalue identity is
\begin{equation}
 D_N^1(q,t)P_\lambda[X_N;q,t]
   =\varepsilon_\lambda(q,t)P_\lambda[X_N;q,t],
 \qquad
 \varepsilon_\lambda(q,t)=\sum_{i=1}^N q^{\lambda_i}t^{N-i},
 \label{eq:macdonald-first-eigenvalue}
\end{equation}
where $\lambda$ is padded by zeros to length $N$; see
\cite[Chapter VI, Sections 3--4]{Macdonald1995}.
In particular, the case $\lambda=\varnothing$ gives
\[
 \sum_{i=1}^N A_i(X_N;t)=1+t+\cdots+t^{N-1}.
\]
For later use, let $<_{\mathrm{lex}}$ denote lexicographic order on
partitions of a fixed size, padded by zeros.  Dominance triangularity
implies lexicographic triangularity, since
$\alpha<\lambda$ in dominance order implies
$\alpha<_{\mathrm{lex}}\lambda$.

The formalization constructs $P_\lambda$ as the unique monic, lexicographically triangular eigenvector of $D_N^1$ with eigenvalue
$\varepsilon_\lambda(q,t)$. This agrees with the standard normalization
used here: dominance-lower terms are lexicographically lower, and the
generic eigenvalues $\varepsilon_\lambda(q,t)$ are pairwise distinct.}

\begin{example}
Take $\lambda=(1)$ and $r=1$. The only possible partitions are $(2)$ and $(1,1)$, so
\[
 P_{(1)}\elementaryE_1
 =\psi'_{(2)/(1)}P_{(2)}
  +\psi'_{(1,1)/(1)}P_{(1,1)}.
\]
For $(2)/(1)$ the strip vector is $\delta=(1,0)$, so no pair of rows contributes and the coefficient is $1$. For $(1,1)/(1)$ the strip vector is $\delta=(0,1)$. The only contributing pair is $(i,j)=(1,2)$, with $\lambda_1-\lambda_2=1$ and $j-i=1$, so
\[
 \psi'_{(1,1)/(1)}(q,t)
 =\frac{(1-q)(1-t^2)}{(1-qt)(1-t)}
 =\frac{(1-q)(1+t)}{1-qt},
\]
and therefore
\[
 P_{(1)}\elementaryE_1
 =P_{(2)}+\frac{(1-q)(1+t)}{1-qt}P_{(1,1)}.
\]
\end{example}

\section{Weighted walks on restricted partitions}\label{sec:walks}
Fix $k\ge2$ and $N\ge1$. All graphs in this section are finite, and multiple edges are distinguished by the object used to define them.

\subsection{Restricted partitions and their decomposition}
\begin{definition}
A partition $\lambda$ is \defin{$k$-restricted} if
\[
 0\le\lambda_i-\lambda_{i+1}<k\qquad(i\ge1).
\]
Let \defin{$\Rset$} be the set of $k$-restricted partitions of length at most $N$.
\end{definition}
The condition includes the last nonzero part. Indeed, if $\ell(\lambda)=s$, then $\lambda_{s+1}=0$, so the inequality at $i=s$ reads $0\le\lambda_s<k$.

Suppose that
\begin{equation}\label{eq:adjacent-differences}
 d_i=\lambda_i-\lambda_{i+1},\qquad 1\le i\le N,
\end{equation}
where $\lambda_{N+1}=0$. If $\lambda$ is $k$-restricted, then each $d_i$ belongs to $\{0,1,\ldots,k-1\}$. Conversely, given arbitrary digits $d_1,\ldots,d_N$ in this set, define
\begin{equation}\label{eq:recover-partition}
 \lambda_i=d_i+d_{i+1}+\cdots+d_N.
\end{equation}
Then $\lambda_i-\lambda_{i+1}=d_i\ge0$, so $\lambda$ is a partition, and it is $k$-restricted because $d_i<k$. Thus
\[
 \lambda\longleftrightarrow(d_1,\ldots,d_N)
\]
is a bijection between $\Rset$ and $\{0,1,\ldots,k-1\}^N$. In particular,
\begin{equation}\label{eq:number-states}
 |\Rset|=k^N.
\end{equation}
We use the partition notation in the statements and the adjacent differences only as a convenient coordinate system in calculations.

\begin{remark}
The restriction is on adjacent differences, not on the parts themselves. For example, when $k=3$, $\lambda=(5,3,1)$ is $3$-restricted because its adjacent differences are $(2,2,1)$, although $\lambda_1=5>3$. 
\end{remark}

\begin{lemma}\label{lem:decomposition}
Every partition $\eta$ of length at most $N$ has a unique expression
\begin{equation}\label{eq:decomposition}
 \eta=\tau+k\xi,\qquad \tau\in\Rset,
\end{equation}
where $\xi$ is a partition of length at most $N$.
\end{lemma}
\begin{proof}
Suppose $\eta_{N+1}=0$. For each $1\le i\le N$, perform Euclidean division on the adjacent difference:
\begin{equation}\label{eq:division-differences}
 \eta_i-\eta_{i+1}=kq_i+r_i,
 \qquad q_i\in\mathbb Z_{\ge0},\quad 0\le r_i<k.
\end{equation}
The quotients are nonnegative because $\eta_i\ge\eta_{i+1}$. Define
\[
 \tau_i=\sum_{j=i}^N r_j,
 \qquad
 \xi_i=\sum_{j=i}^N q_j.
\]
Then
\[
 \tau_i-\tau_{i+1}=r_i\ge0,
 \qquad
 \xi_i-\xi_{i+1}=q_i\ge0,
\]
so both $\tau$ and $\xi$ are partitions. Moreover $0\le\tau_i-\tau_{i+1}=r_i<k$, so $\tau$ is $k$-restricted.

It remains to verify the equality $\eta=\tau+k\xi$. For every $i$,
\begin{align*}
 (\tau_i+k\xi_i)-(\tau_{i+1}+k\xi_{i+1})
 =r_i+kq_i=\eta_i-\eta_{i+1}.
\end{align*}
Both sequences vanish in position $N+1$. Equality of all adjacent differences together with equality of the final zero term therefore implies
\[
 \eta_i=\tau_i+k\xi_i\qquad(1\le i\le N).
\]

For uniqueness, suppose also $\eta=\widetilde\tau+k\widetilde\xi$ with $\widetilde\tau$ $k$-restricted. Taking adjacent differences gives
\[
 \eta_i-\eta_{i+1}
 = (\widetilde\tau_i-\widetilde\tau_{i+1})
   +k(\widetilde\xi_i-\widetilde\xi_{i+1}),
\]
where $0\le\widetilde\tau_i-\widetilde\tau_{i+1}<k$. By uniqueness of Euclidean division, the first term must equal $r_i$ and the second quotient must equal $q_i$. Hence $\widetilde\tau=\tau$ and $\widetilde\xi=\xi$ after summing the adjacent differences, and we are done.
\end{proof}

\begin{remark}
The unique decomposition is obtained by Euclidean division of the adjacent differences
 $\eta_i-\eta_{i+1}$, not by reducing each part $\eta_i$ separately modulo $k$. The latter operation may even destroy the partition inequalities. 

For instance, with $k=3$ and $\eta=(4,2)$, reducing the parts separately gives $(1,2)$, which is not a partition. The correct adjacent differences are $(2,2)$; both quotients are zero and both remainders are $2$, so the decomposition is simply $\tau=(4,2)$ and $\xi=\varnothing$. Thus the ``modulo $k$ reduction'' in this paper always refers to adjacent row differences.    
\end{remark}

\begin{example}\label{ex:decomposition}
Let $k=3$ and $\eta=(7,4,1)$.
The adjacent differences are $(3,3,1)$.
Division by $3$ gives the quotients $q=(1,1,0)$
and the remainders $r=(0,0,1)$.
Thus
\[
 \tau=(r_1+r_2+r_3,r_2+r_3,r_3)=(1,1,1),\qquad
 \xi=(q_1+q_2+q_3,q_2+q_3,q_3)=(2,1,0).
\]
Hence $(7,4,1)=(1,1,1)+3(2,1)$.
\end{example}

\subsection{Edges and labels}
Introduce independent variables $\uvec=(u_1,\ldots,u_N)$, and put $u_\alpha=\prod_ju_{\alpha_j}$ for $\alpha_1\le N$, with $u_\varnothing=1$.

\begin{definition}\label{def:edge}
An \defin{edge} of size $r$ starting at $\lambda\in\Rset$ is specified by a partition $\eta$ of length at most $N$ such that $\eta/\lambda$ is a vertical $r$-strip. Apply Lemma~\ref{lem:decomposition} to write $\eta=\tau+k\xi$. The edge ends at $\tau$ and has \defin{label} $\xi'$, or equivalently the monomial $u_{\xi'}$. Distinct choices of $\eta$ give distinct edges, even if their endpoints and labels agree.
\end{definition}

An edge of the graph is directed from a $k$-restricted partition
$\lambda$ to a $k$-restricted partition $\tau$. It is obtained by
first adding a vertical strip to form an intermediate partition
$\eta$, and then writing
\[
\eta=\tau+k\xi.
\]
Thus the actual edge is $\lambda\to\tau$, while $\eta$ is intermediate
and $\xi'$ determines the label $u_{\xi'}$.

Since $\ell(\xi)\le N$, every part of $\xi'$ is at most $N$, so the monomial $u_{\xi'}$ is defined. There are no edges of size $r>N$, because a vertical strip in at most $N$ rows contains at most one new cell in each row.

The following coordinate description will be used repeatedly.
\begin{lemma}\label{lem:edge-coordinates}
Let $\lambda\in\Rset$ and write $d_i=\lambda_i-\lambda_{i+1}$. A vertical strip $\eta/\lambda$ is uniquely specified by a vector
\[
 \delta=(\delta_1,\ldots,\delta_N)\in\{0,1\}^N,
 \qquad \eta_i=\lambda_i+\delta_i.
\]
Its size is $r=\sum_i\delta_i$. let $\delta_{N+1}=0$ and
\begin{equation}\label{eq:delta}
 g_i=d_i+\delta_i-\delta_{i+1}.
\end{equation}
Although $g_i$ is first defined from the old-state coordinates $d_i$ and the strip vector $\delta$, it is simultaneously the adjacent row difference of the intermediate partition $\eta$:
\[
 g_i=\eta_i-\eta_{i+1}
 = (\lambda_i-\lambda_{i+1})+(\delta_i-\delta_{i+1}).
\]
The first expression is convenient for computing from the old state; the second explains why reducing $g_i$ modulo $k$ is exactly the decomposition of Lemma~\ref{lem:decomposition}.
Then, we have
\begin{enumerate}
\item $\eta$ is a partition if and only if $g_i\ge0$ for all $i$;
\item for a legal strip, $0\le g_i\le k$;
\item if $\eta=\tau+k\xi$ is the decomposition of Lemma~\ref{lem:decomposition}, then
\[
 \tau_i-\tau_{i+1}\equiv g_i\pmod k,
 \qquad
 \xi_i-\xi_{i+1}=\begin{cases}1,&g_i=k,\\0,&g_i<k;\end{cases}
\]
\item the parts of $\xi'$ are exactly the indices $i$ for which $g_i=k$, and hence
\begin{equation}\label{eq:labelcoordinates}
 u_{\xi'}=\prod_{i:g_i=k}u_i;
\end{equation}
\item the sizes satisfy
\begin{equation}\label{eq:size}
 |\tau|-|\lambda|=r-k|\xi|.
\end{equation}
\end{enumerate}
\end{lemma}
\begin{proof}
Because $\eta_i=\lambda_i+\delta_i$,
\[
 \eta_i-\eta_{i+1}
 = (\lambda_i-\lambda_{i+1})+(\delta_i-\delta_{i+1})
 =g_i.
\]
Thus $\eta$ is weakly decreasing exactly when every $g_i$ is nonnegative, proving (1).

Since $\lambda$ is restricted, $0\le d_i\le k-1$. Also $\delta_i-\delta_{i+1}\in\{-1,0,1\}$. Therefore $g_i\le k$. For a legal strip we already know $g_i\ge0$, which proves (2).

Now apply Euclidean division to $g_i$. Because $0\le g_i\le k$, its quotient upon division by $k$ is either $0$ or $1$: it is $1$ exactly when $g_i=k$. Lemma~\ref{lem:decomposition} identifies these quotients with the adjacent differences $\xi_i-\xi_{i+1}$ and the remainders with $\tau_i-\tau_{i+1}$. This gives (3).

More explicitly,
\[
 g_i=(\tau_i-\tau_{i+1})+k(\xi_i-\xi_{i+1}),
\]
with $0\le\tau_i-\tau_{i+1}<k$. Thus $\tau_i-\tau_{i+1}$ is literally the remainder of $g_i$ upon division by $k$, while $\xi_i-\xi_{i+1}$ is the quotient. Since a legal strip gives $0\le g_i\le k$, only two cases can occur:
\[
 g_i<k\quad\Longrightarrow\quad
 \xi_i-\xi_{i+1}=0,\qquad
 \tau_i-\tau_{i+1}=g_i,
\]
and
\[
 g_i=k\quad\Longrightarrow\quad
 \xi_i-\xi_{i+1}=1,\qquad
 \tau_i-\tau_{i+1}=0.
\]
One may therefore think of $g_i=k$ as an adjacent difference that has just filled one complete block of size $k$; that block is factored off, and the positions at which this happens become the parts of the edge label $\xi'$.

For every partition $\xi$, the multiplicity of the part $i$ in $\xi'$ is
\[
 m_i(\xi')=\xi_i-\xi_{i+1}.
\]
Hence (3) says that $i$ occurs as a part of $\xi'$ precisely when $g_i=k$, and it occurs once. This proves (4). In particular, the label of a single edge is squarefree, although the same index may occur in several different steps of a walk.

Finally, $|\eta|=|\lambda|+r$ because the strip has $r$ cells, while $|\eta|=|\tau|+k|\xi|$ by the decomposition. Equating these two expressions and rearranging gives (5).
\end{proof}

\begin{example}\label{ex:edge}
We illustrate explicitly how an edge and its label are obtained.
Take $k=3$, $N=2$, and $\lambda=(3,1)$. With the convention $\lambda_3=0$, we have $d=(2,1)$.

Consider a step of size $1$ obtained by adding one cell to the first
row. Thus
\[
\delta=(1,0),
\qquad
\eta=\lambda+\delta=(4,1).
\]
The new adjacent differences are $g_i=d_i+\delta_i-\delta_{i+1}, \delta_3=0$,
thus $g=(3,1)$.
Since $g_1=k=3$, reducing the adjacent differences modulo $k$ gives
\[
\tau_1-\tau_2=0,
\qquad
\tau_2-\tau_3=1,\qquad
\xi_1-\xi_2=1,
\qquad
\xi_2-\xi_3=0.
\]
Therefore $\tau=(1,1), \xi=(1)$,
and indeed $(4,1)=(1,1)+3(1)$. 
Hence the corresponding edge has initial vertex $\lambda=(3,1)$ and
final vertex $\tau=(1,1)$. Since $\xi'=(1)$,
its label is $u_{\xi'}=u_1$. Thus,
\[
(3,1)\longrightarrow(1,1),
\qquad\text{with label }u_1.
\]
By contrast, for $\lambda=(1,1)$ the choice $\delta=(0,1)$ gives $\eta=(1,2)$, which is not a partition. It does not define an edge.
\end{example}

\subsection{Weights}\label{subsection:weights}
The preceding subsection determines the directed graph and the labels. The weights come from the vertical Pieri rule~\eqref{eq:pieri}. Let $\zeta$ be a primitive $k$th root of unity and write
\[
 q=\zeta,\qquad t=\zeta v.
\]
Fix a contributing pair $(i,j)$ in~\eqref{eq:rowpairs}, and put
\[
 D=\lambda_i-\lambda_j\ge1,\qquad m=j-i\ge1.
\]
Since $q^at^b=\zeta^{a+b}v^b$, the factor of this pair becomes
\begin{equation}\label{eq:weight-motivation}
 \frac{1-\zeta^{D+m-1}v^{m-1}}{1-\zeta^{D+m-1}v^{m}}
 \cdot
 \frac{1-\zeta^{D+m}v^{m+1}}{1-\zeta^{D+m}v^{m}}.
\end{equation}
Both quotients have the form $(1-\zeta^av^b)/(1-\zeta^av^m)$ with $b\in\{m-1,m+1\}$. After cancellation, evaluation at $v=1$ gives $1$ when $k\nmid a$, since then $\zeta^a\ne1$. When $k\mid a$, the identity
\[
 \frac{1-v^{b}}{1-v^m}
 =\frac{1+v+\cdots+v^{b-1}}{1+v+\cdots+v^{m-1}}
\]
shows that the value at $v=1$ is $b/m$. This motivates the following definition.

For integers $a$, $b\ge0$ and $m\ge1$, let
\begin{equation}\label{eq:beta}
 \defin{\beta_k(a;b,m)}\coloneqq
 \begin{cases}
  1,&k\nmid a,\\[2pt]
  b/m,&k\mid a.
 \end{cases}
\end{equation}
For an edge given by the strip $\eta=\lambda+\delta$, let
\begin{equation}\label{eq:weight}
 \defin{w_k(\eta/\lambda)}\coloneqq
 \prod_{\substack{1\le i<j\le N\\\delta_i=0,\ \delta_j=1}}
 \beta_k(D+m-1;\,m-1,m)\,\beta_k(D+m;\,m+1,m),
 \qquad D=\lambda_i-\lambda_j,\ m=j-i.
\end{equation}
An empty product is $1$.

\begin{lemma}\label{lem:weights}
For every edge, $w_k(\eta/\lambda)\in\QQ_{>0}$.
\end{lemma}
\begin{proof}
Each factor in~\eqref{eq:weight} equals $1$, $(m+1)/m$ or $(m-1)/m$, and is therefore nonnegative. It vanishes only if $m=1$ and $k\mid D$. In that case $j=i+1$, and $1\le D=\lambda_i-\lambda_{i+1}\le k-1$ because the strip is legal and $\lambda$ is $k$-restricted. Hence $k\nmid D$, and every factor is positive.
\end{proof}

The proof shows that the weights are nonnegative for every partition $\lambda$, restricted or not; restriction is used only to make them positive.

\begin{example}\label{ex:weight}
Let $k=3$, $\lambda=(3,1,0)$ and $\eta=(3,2,1)$, so $\delta=(0,1,1)$. The contributing pairs are $(1,2)$ and $(1,3)$. For $(1,2)$ we have $D=2$ and $m=1$, which gives the factors $\beta_3(2;0,1)=1$ and $\beta_3(3;2,1)=2$. For $(1,3)$ we have $D=3$ and $m=2$, which gives $\beta_3(4;1,2)\beta_3(5;3,2)=1$. Consequently $w_3(\eta/\lambda)=2$.

The weights need not be integers. For $k=3$, $\lambda=(1,1,0)$ and $\eta=(1,1,1)$ we have $\delta=(0,0,1)$. The pair $(1,3)$ has $D=1$ and $m=2$, which gives $\beta_3(2;1,2)\beta_3(3;3,2)=3/2$, and the pair $(2,3)$ has $D=1$ and $m=1$, which gives the factor $1$. So $w_3(\eta/\lambda)=3/2$.
\end{example}

\subsection{Closed walks}
A \defin{walk} with successive step sizes $r_1,\ldots,r_s$ consists of edges
\[
 \lambda^0\xrightarrow{\eta^1}\lambda^1
 \xrightarrow{\eta^2}\cdots
 \xrightarrow{\eta^s}\lambda^s,
 \qquad \eta^j=\lambda^j+k\xi^j,
\]
where $\eta^j/\lambda^{j-1}$ is a vertical $r_j$-strip. It is \defin{closed} if $\lambda^s=\lambda^0$. The initial vertex is part of the data; walks are not identified under cyclic rotation. Define
\begin{equation}\label{eq:walkdata}
 \defin{\wt(W)}\coloneqq\prod_{j=1}^s w_k(\eta^j/\lambda^{j-1}),
 \qquad
 \defin{u_{\nu(W)}}\coloneqq\prod_{j=1}^s u_{(\xi^j)'}.
\end{equation}
Equivalently, $\nu(W)$ is obtained by collecting the parts of all $(\xi^j)'$ and sorting them in decreasing order.

For example, if the three edge labels are $(3,1)$, $(2)$, and $(3,2)$, then
\[
 (u_3u_1)u_2(u_3u_2)=u_3^2u_2^2u_1,
\]
so the combined label is
\[
 \nu(W)=(3,3,2,2,1).
\]
The operation is simply concatenation of all label parts; it does not alter the numerical edge weights.

Applying \eqref{eq:size} to each step gives
\[
 |\lambda^j|-|\lambda^{j-1}|=r_j-k|\xi^j|.
\]
Summing over $j$ yields
\[
 |\lambda^s|-|\lambda^0|
 =\sum_{j=1}^s r_j-k\sum_{j=1}^s|\xi^j|.
\]
If $W$ is closed, the left-hand side is zero.
Conjugation preserves size, so
\[
 |\nu(W)|
 =\sum_{j=1}^s|(\xi^j)'|
 =\sum_{j=1}^s|\xi^j|.
\]
Therefore
\begin{equation}\label{eq:walkdegree}
 k|\nu(W)|=\sum_{j=1}^s r_j.
\end{equation}
In particular, let
$\mu=(\mu_1,\ldots,\mu_s)\vdash kn$ and set $r_j=\mu_j$.
Then $|\nu(W)|=n$, so $\nu(W)\vdash n$.
 
 Let \defin{$\mathcal W_{k,N}(\mu,\nu)$} be the finite set of these closed walks with label $\nu$. We will prove
\begin{equation}\label{eq:walkformula}
 c^{(k)}_{\mu\nu}=\frac{1}{k^N}
       \sum_{W\in\mathcal W_{k,N}(\mu,\nu)}\wt(W)
       \qquad(N\ge n).
\end{equation}
The graph and all terms on the right have now been defined independently of the coefficients on the left.

\section{Macdonald polynomials at a root of unity}\label{sec:macdonald}
The ordinary Macdonald identities needed below were collected in Section~\ref{subsec:macdonald-background}. We now specialize the first parameter to a primitive $k$th root of unity $\zeta$, while keeping $t$ indeterminate. The purpose of this section is to prove the factorization that turns an intermediate partition $\eta$ into the finite-state endpoint $\tau$ and the label $\xi'$.

\subsection{The Pieri coefficients at the root of unity}
\begin{lemma}\label{lem:rowpairs}
Let $\eta/\lambda$ be a vertical strip with $\ell(\eta)\le N$. Then $\psi'_{\eta/\lambda}(q,t)$ is regular at $q=\zeta$, with $t$ indeterminate, whether or not $\lambda$ is restricted.
\end{lemma}
\begin{proof}
Every pair $(i,j)$ in~\eqref{eq:rowpairs} has $j-i\ge1$, and $\lambda_i>\lambda_j$ by the remark following the definition of vertical strips. In each denominator factor the power of $t$ is $j-i\ge1$. After setting $q=\zeta$, each denominator factor is therefore of the form $1-\zeta^at^b$ with $b\ge1$, which is not the zero polynomial in $t$.
\end{proof}

\subsection{Regularity at the root of unity}
\begin{lemma}\label{lem:regular}
The coefficients of $P_\lambda[X;q,t]$ are regular at $q=\zeta$ with $t$ indeterminate. The polynomials $P_\lambda[X_N;\zeta,t]$, $\ell(\lambda)\le N$, form a basis of $\Lambda_N(\QQ(\zeta,t))$.
\end{lemma}
{\Lilan
\begin{proof}
Let $\mathcal O\subseteq\QQ(q,t)$ be the ring of rational functions
regular at $q=\zeta$, with $t$ indeterminate.  Lemma~\ref{lem:rowpairs} shows that every vertical Pieri coefficient lies in $\mathcal O$
(choose $N$ large enough for the partitions in that coefficient): each denominator factor specializes to a nonzero polynomial
$1-\zeta^a t^b$, with $a\geq0$ and $b\geq1$.

For $d\geq0$, let
\[
 S_d=\sum_{\lambda\vdash d}\mathcal O\,P_\lambda
 \subseteq\Lambda_d(\QQ(q,t)).
\]
The generic Pieri rule gives $\elementaryE_rS_d\subseteq S_{d+r}$.
Since $P_\varnothing=1$, it follows that $\elementaryE_\nu\in S_d$ whenever
$|\nu|=d$.  The elementary and monomial symmetric functions are $\ZZ$-bases of $\Lambda_{\ZZ}$; hence each $\monomial_\lambda$,
$\lambda\vdash d$, is an integer linear combination of these $\elementaryE_\nu$
and therefore belongs to $S_d$.

Order the $p(d)$ partitions of $d$ by a linear extension of dominance, and let $\mathbf m$ and $\mathbf P$ be the corresponding column vectors.
The preceding inclusion gives
\[
 \mathbf m=A\mathbf P,
 \qquad A\in\operatorname{Mat}_{p(d)}(\mathcal O).
\]
The defining monomial expansions give $\mathbf P=T\mathbf m$ with $T$ unitriangular over $\QQ(q,t)$.  Thus $A=T^{-1}$ is also unitriangular. Write $A=I+B$, with $B$ strictly triangular. Since $B^{p(d)}=0$, the matrix
\[
 A^{-1}=\sum_{j=0}^{p(d)-1}(-B)^j
\]
also has entries in $\mathcal O$. Thus $\mathbf P=A^{-1}\mathbf m$ has all its coefficients in $\mathcal O$. This proves regularity in every degree.

After specialization, the monomial expansion of $P_\lambda$ remains unitriangular, with diagonal coefficient $1$. In $N$ variables the terms indexed by partitions of length greater than $N$ vanish; in each
degree the remaining transition matrix is still unitriangular and is therefore invertible over $\QQ(\zeta,t)$.
\end{proof}

}

\subsection{Rectangular partitions}
\begin{lemma}\label{lem:rectangle}
For $0\le h\le N$,
\begin{equation}\label{eq:rectangle}
 P_{(k^h)}[X_N;\zeta,t]=E_h(X_N),
\end{equation}
where $E_0=1$.
\end{lemma}
{\Lilan
\begin{proof}
The case $h=0$ is immediate. Suppose $1\le h\le N$, and let
\[
 \rho=(k^h),\qquad M=\monomial_\rho[X_N]=\elementaryE_h(x_1^k,\ldots,x_N^k)=E_h(X_N),
 \qquad c_N(t)=\sum_{i=1}^N t^{N-i}.
\]
Since $\zeta^k=1$, every shift $T_{\zeta,x_i}$ fixes $M$. Hence
\[
 D_N^1(\zeta,t)M=c_N(t)M.
\]
By Lemma~\ref{lem:regular}, the coefficients of every
$P_\alpha[X_N;q,t]$ are regular at $q=\zeta$. We may therefore specialize
\eqref{eq:macdonald-first-eigenvalue} to obtain
\[
 D_N^1(\zeta,t)P_\alpha[X_N;\zeta,t]
  =\varepsilon_\alpha(\zeta,t)P_\alpha[X_N;\zeta,t].
\]
In particular, $\varepsilon_\rho(\zeta,t)=c_N(t)$.

The specialized Macdonald polynomials form a unitriangular basis in
each degree. Inverting their triangular expansion in the monomial
basis gives
\[
 P_\rho[X_N;\zeta,t]-M
   =\sum_{\substack{\alpha\vdash kh,\ \ell(\alpha)\le N\\
                   \alpha<_{\mathrm{lex}}\rho}}
       a_\alpha P_\alpha[X_N;\zeta,t],
 \qquad a_\alpha\in\QQ(\zeta,t).
\]
We claim that $\varepsilon_\alpha(\zeta,t)\ne c_N(t)$ for every index in this sum.
Otherwise, because $t$ is indeterminate, comparing coefficients in
\[
 \sum_{i=1}^N\zeta^{\alpha_i}t^{N-i}
   =\sum_{i=1}^N t^{N-i}
\]
would give $\zeta^{\alpha_i}=1$ for every $i$.  Primitivity of $\zeta$
then implies that every $\alpha_i$ is divisible by $k$. But
$\alpha<_{\mathrm{lex}}(k^h)$ implies $\alpha_1\le k$, so every part
of $\alpha$ is either $0$ or $k$. Since $|\alpha|=kh$, this forces
$\alpha=(k^h)$, a contradiction.

Apply $D_N^1(\zeta,t)-c_N(t)\operatorname{id}$ to the displayed expansion of $P_\rho-M$. Its left-hand side vanishes, giving
\[
 \sum_{\substack{\alpha\vdash kh,\ \ell(\alpha)\le N\\
\alpha<_{\mathrm{lex}}\rho}}
 a_\alpha\bigl(\varepsilon_\alpha(\zeta,t)-c_N(t)\bigr)
 P_\alpha[X_N;\zeta,t]=0.
\]
Linear independence and the claim imply $a_\alpha=0$ for every
$\alpha$. Thus $P_\rho[X_N;\zeta,t]=M=E_h(X_N)$.
\end{proof}
}

\subsection{Adding a rectangle}
For $1\le h\le N$, let $\lambda^+=\lambda+k(1^h)$. Thus the first $h$ parts increase by $k$.

\begin{lemma}\label{lem:periodicity}
Let $\delta\in\{0,1\}^N$.
If $\lambda+\delta$ is a partition, then
$\lambda^++\delta$ is also a partition, and
\begin{equation}\label{eq:periodicity}
 \psi'_{\lambda^++\delta\,/\,\lambda^+}(\zeta,t)
 =
 \psi'_{\lambda+\delta\,/\,\lambda}(\zeta,t).
\end{equation}
If $\lambda^++\delta$ is a partition but
$\lambda+\delta$ is not, then $\psi'_{\lambda^++\delta\,/\,\lambda^+}(\zeta,t)=0$.
\end{lemma}
\begin{proof}
We first compare strips that are already legal over $\lambda$. Write
\[
 \lambda_i^+=\begin{cases}\lambda_i+k,&i\le h,\\\lambda_i,&i>h.\end{cases}
\]
If $\eta=\lambda+\delta$ is a partition, then
\[
 \eta^+=\lambda^++\delta=\eta+k(1^h)
\]
is also a partition, because adding the same number $k$ to the first $h$ parts preserves all weak inequalities. Thus every old strip remains legal.

For such a common strip, the conditions $\delta_i=0$ and $\delta_j=1$ indexing~\eqref{eq:rowpairs} are unchanged. For each pair $i<j$,
\[
 \lambda_i^+-\lambda_j^+
 =\begin{cases}
   \lambda_i-\lambda_j,& i,j\le h\text{ or }i,j>h,\\
   \lambda_i-\lambda_j+k,& i\le h<j.
  \end{cases}
\]
Every occurrence of $q^{\lambda_i-\lambda_j}$ or $q^{\lambda_i-\lambda_j-1}$ is therefore multiplied by either $1$ or $q^k$. After setting $q=\zeta$ and using $\zeta^k=1$, every row-pair factor is unchanged. This proves~\eqref{eq:periodicity} for the common strips.

We next determine whether a vector $\delta$ can be illegal over $\lambda$ but legal over $\lambda^+$. The adjacent differences of $\lambda^+$ agree with those of $\lambda$ except at the boundary $h$:
\[
 \lambda_i^+-\lambda_{i+1}^+
 =\begin{cases}
 \lambda_i-\lambda_{i+1},&i\ne h,\\
 \lambda_h-\lambda_{h+1}+k,&i=h.
 \end{cases}
\]
Hence all partition inequalities except possibly the one between rows $h$ and $h+1$ are identical before and after translation. If $\lambda+\delta$ fails to be a partition but $\lambda^++\delta$ is a partition, the failure must occur at this boundary. The bad inequality is
\[
 \lambda_h+\delta_h<\lambda_{h+1}+\delta_{h+1}.
\]
Because $\lambda_h\ge\lambda_{h+1}$ and $\delta_h,\delta_{h+1}\in\{0,1\}$, this can happen only when
\[
 \lambda_h=\lambda_{h+1},
 \qquad \delta_h=0,
 \qquad \delta_{h+1}=1.
\]
This also shows that $h<N$.

For such a newly legal strip, the pair $(h,h+1)$ occurs in the row-pair product because $\delta_h=0$ and $\delta_{h+1}=1$. In $\lambda^+$ the corresponding row difference equals
\[
 \lambda_h^+-\lambda_{h+1}^+=k.
\]
The first numerator factor in~\eqref{eq:rowpairs} is therefore
\[
 1-q^k t^{(h+1)-h-1}=1-q^k.
\]
It vanishes at $q=\zeta$. By the regularity statement in Lemma~\ref{lem:rowpairs}, none of the denominator factors becomes the zero polynomial in $t$, so this zero cannot be cancelled by a pole. The whole Pieri coefficient is therefore zero at $q=\zeta$.
\end{proof}

\begin{example}\label{ex:newstrip}
Take $k=2$, $\lambda=(1,1)$, and $h=1$. The choice $\delta=(0,1)$ is not a strip over $\lambda$, but it is a strip over $\lambda^+=(3,1)$. Its coefficient is
\[
 \psi'_{(3,2)/(3,1)}(q,t)
 =\frac{(1-q^2)(1-qt^2)}{(1-q^2t)(1-qt)},
\]
which is zero at $q=-1$ with $t$ indeterminate.
\end{example}

\begin{proposition}\label{prop:shift}
For every $\lambda$ of length at most $N$ and $1\le h\le N$,
\begin{equation}\label{eq:shift}
 P_{\lambda+k(1^h)}[X_N;\zeta,t]
     =E_h(X_N)P_\lambda[X_N;\zeta,t].
\end{equation}
\end{proposition}
\begin{proof}
By Lemma~\ref{lem:regular}, the specialized polynomials
\[
 \{P_\rho[X_N;\zeta,t]:\ell(\rho)\le N\}
\]
form a basis of $\Lambda_N(\QQ(\zeta,t))$. We may therefore define a linear operator $T_h$ on this basis by
\[
 \defin{T_h(P_\rho)}\coloneqq P_{\rho+k(1^h)}.
\]
The partition $\rho+k(1^h)$ still has at most $N$ parts, so the definition stays inside the same space.

Fix $1\le r\le N$ and a basis vector $P_\lambda$. The Pieri rule gives
\begin{equation}\label{eq:Th-left}
 \elementaryE_rP_\lambda
 =\sum_{\eta/\lambda\text{ vertical }r\text{-strip}}
 \psi'_{\eta/\lambda}(\zeta,t)P_\eta.
\end{equation}
Applying $T_h$ term by term yields
\[
 T_h(\elementaryE_rP_\lambda)
 =\sum_{\eta/\lambda\text{ vertical }r\text{-strip}}
 \psi'_{\eta/\lambda}(\zeta,t)P_{\eta+k(1^h)}.
\]
On the other hand,
\begin{equation}\label{eq:Th-right}
 \elementaryE_rT_h(P_\lambda)
 =\elementaryE_rP_{\lambda+k(1^h)}
 =\sum_{\theta/(\lambda+k(1^h))\text{ vertical }r\text{-strip}}
 \psi'_{\theta/(\lambda+k(1^h))}(\zeta,t)P_\theta.
\end{equation}
Every old strip $\eta/\lambda$ produces the strip
\[
 \theta=\eta+k(1^h)
\]
over $\lambda+k(1^h)$, and Lemma~\ref{lem:periodicity} says that the two coefficients are equal. Conversely, any strip in~\eqref{eq:Th-right} that does not arise in this way is newly legal after translation, and Lemma~\ref{lem:periodicity} says that its coefficient is zero. Thus the two expansions agree term by term, proving
\begin{equation}\label{eq:commute-er}
 T_h(\elementaryE_rP_\lambda)=\elementaryE_rT_h(P_\lambda)
 \qquad(1\le r\le N).
\end{equation}

Because the $P_\lambda$ form a basis, linearity extends~\eqref{eq:commute-er} from basis vectors to every symmetric polynomial $g$:
\[
 T_h(\elementaryE_rg)=\elementaryE_rT_h(g).
\]
The ring $\Lambda_N$ is the polynomial algebra in $\elementaryE_1,\ldots,\elementaryE_N$. Repeatedly applying the preceding identity shows that for every polynomial $f(\elementaryE_1,\ldots,\elementaryE_N)$ and every $g$,
\[
 T_h(fg)=fT_h(g).
\]
Taking $g=1$ gives $T_h(f)=fT_h(1)$.
By definition,
\[
 T_h(1)=T_h(P_\varnothing)=P_{(k^h)},
\]
and Lemma~\ref{lem:rectangle} identifies this polynomial with $E_h$. Hence $ T_h(f)=fE_h$.
Finally, let $f=P_\lambda$. The left side is $P_{\lambda+k(1^h)}$ by definition of $T_h$, while the right side is $E_hP_\lambda$. This is~\eqref{eq:shift}.
\end{proof}

\begin{corollary}\label{cor:factorization}
For the decomposition $\eta=\tau+k\xi$ of Lemma~\ref{lem:decomposition},
\begin{equation}\label{eq:factorization}
 P_\eta[X_N;\zeta,t]=\elementaryE_{\xi'}[p_k](X_N)P_\tau[X_N;\zeta,t].
\end{equation}
\end{corollary}
\begin{proof}
Let $q_i=\xi_i-\xi_{i+1}, (1\le i\le N)$, with $\xi_{N+1}=0$. These are nonnegative integers. Since
\[
 (1^i)=(\underbrace{1,\ldots,1}_{i\text{ entries}},0,\ldots,0),
\]
summing adjacent differences gives the standard decomposition
\begin{equation}\label{eq:xi-columns}
 \xi=\sum_{i=1}^N q_i(1^i).
\end{equation}
Equivalently, $q_i=m_i(\xi')$, the multiplicity of the part $i$ in the conjugate partition $\xi'$.

Starting from $\tau$, equation~\eqref{eq:xi-columns} gives
\[
 \eta=\tau+k\xi
 =\tau+\sum_{i=1}^N q_i k(1^i).
\]
Apply Proposition~\ref{prop:shift} once for each summand $k(1^i)$, repeating the height $i$ shift exactly $q_i$ times. The order is irrelevant because multiplication by the symmetric polynomials $E_i$ commutes. We obtain
\[
 P_\eta
 =\prod_{i=1}^N E_i^{q_i}P_\tau.
\]
Since $q_i=m_i(\xi')$,
\[
 \prod_i E_i^{q_i}
 =\prod_i \elementaryE_i[p_k]^{m_i(\xi')}
 =e_{\xi'}[p_k].
\]
This proves~\eqref{eq:factorization}.
\end{proof}
For the decomposition in Example~\ref{ex:decomposition}, this reads
\[
 P_{(7,4,1)}[X_N;\zeta,t]=E_2(X_N)E_1(X_N)\elementaryE_3[X_N]
 \qquad(k=3,\ N\ge3).
\]
The factor $e_{\xi'}=e_{21}$ is exactly the label assigned to the corresponding decomposition in Section~\ref{sec:walks}.

\begin{remark}
Descouens and Morita proved a factorization for $Q'_\lambda$
at roots of unity
\cite[Theorem~5.1]{DescouensMorita2008}.
Their normalization is
\[
 Q'_\lambda[X;q,t]
 =
 Q_\lambda\left[\frac{1-q}{1-t}X;q,t\right].
\]
Let $\omega$ be the standard involution defined by
$\omega(\elementaryE_r)=\completeH_r$.
Macdonald duality gives
\[
 P_\lambda[X;q,t]
 =
 \omega\bigl(Q'_{\lambda'}[X;t,q]\bigr).
\]
Applying this identity to their factorization gives Corollary~\ref{cor:factorization}.
We include a direct proof adapted to $\Rset$.
\end{remark}

\section{Interpolation and the root of unity filter}\label{sec:evaluation}
We now relate the weighted graph to multiplication by elementary symmetric polynomials. We first keep an auxiliary parameter and then specialize its matrix entries to the positive weights.

\subsection{Transition matrices and their limit}
Set $t=\zeta v$, where $v$ is an indeterminate. The rows and columns of our matrices are indexed by $\Rset$, with the column denoting the initial vertex. Define
\begin{equation}\label{eq:matrix}
\defin{(M_r(v;\uvec))_{\tau,\lambda}}\coloneqq
 \sum_{\substack{\eta/\lambda\text{ vertical }r\text{-strip}\\
                 \ell(\eta)\le N,\ \eta=\tau+k\xi}}
    \psi'_{\eta/\lambda}(\zeta,\zeta v)u_{\xi'}.
\end{equation}
The decomposition in the sum is always the one from Lemma~\ref{lem:decomposition}. Let $M_r(\uvec)$ be the matrix defined by the same formula with $\psi'_{\eta/\lambda}(\zeta,\zeta v)$ replaced by $w_k(\eta/\lambda)$. Thus $M_r(\uvec)$ is the weighted adjacency matrix for edges of size $r$. For $r>N$, both matrices are zero.

We use column vectors. Thus the column index $\lambda$ is the initial vertex and the row index $\tau$ is the final vertex. Consequently a product
\[
 M_{r_s}\cdots M_{r_1}
\]
represents first taking a step of size $r_1$, then a step of size $r_2$, and so on. This convention is the reason the factors appear in reverse order relative to the chronological order of the walk.

Recall the Pieri expansion
\[
 \elementaryE_r[X_N]P_\lambda[X_N;\zeta,\zeta v]
 =\sum_{\eta/\lambda\text{ vertical }r\text{-strip}}
 \psi'_{\eta/\lambda}(\zeta,\zeta v)
 P_\eta[X_N;\zeta,\zeta v].
\]
For each intermediate partition $\eta$, apply Lemma~\ref{lem:decomposition} to write
\[
 \eta=\tau+k\xi,
 \qquad \tau\in\Rset.
\]
Corollary~\ref{cor:factorization} then gives
\[
 P_\eta[X_N;\zeta,\zeta v]
 =\elementaryE_{\xi'}[p_k](X_N)P_\tau[X_N;\zeta,\zeta v]
 =E_{\xi'}(X_N)P_\tau[X_N;\zeta,\zeta v].
\]
Substituting this into the Pieri expansion and grouping all terms with the same endpoint $\tau$ yields
\begin{align*}\label{eq:multiplication}
e_r[X_N]P_\lambda[X_N;\zeta,\zeta v]
&=
\sum_{\eta/\lambda\text{ vertical }r\text{-strip}}
\psi'_{\eta/\lambda}(\zeta,\zeta v)\,
P_\eta[X_N;\zeta,\zeta v]
\\
&=
\sum_{\eta/\lambda\text{ vertical }r\text{-strip}}
\psi'_{\eta/\lambda}(\zeta,\zeta v)\,
E_{\xi'}(X_N)\,
P_\tau[X_N;\zeta,\zeta v]
\\
&=
\sum_{\tau\in\mathcal R_{k,N}}
\left(
\sum_{\substack{
\eta/\lambda\text{ vertical }r\text{-strip}\\
\eta=\tau+k\xi
}}
\psi'_{\eta/\lambda}(\zeta,\zeta v)\,
E_{\xi'}(X_N)
\right)
P_\tau[X_N;\zeta,\zeta v]
\\
&=
\sum_{\tau\in\mathcal R_{k,N}}
\bigl(M_r(v;E(X_N))\bigr)_{\tau,\lambda}\,
P_\tau[X_N;\zeta,\zeta v].
\end{align*}
That is, 
\begin{equation}\label{eq:multiplication}
 \elementaryE_r[X_N]P_\lambda[X_N;\zeta,\zeta v]
  =\sum_{\tau\in\Rset}
   (M_r(v;E(X_N)))_{\tau,\lambda}
       P_\tau[X_N;\zeta,\zeta v],
\end{equation}
where $E(X_N)=(E_1(X_N),\ldots,E_N(X_N))$. Thus the formal monomial $u_{\xi'}$ in~\eqref{eq:matrix} is not an auxiliary decoration: after the substitution $u_i=E_i(X_N)$, it is exactly the elementary factor produced by the root of unity factorization of the intermediate Macdonald polynomial.

Replacing each Pieri coefficient
$\psi'_{\eta/\lambda}(\zeta,\zeta v)$ by its {\Lilan value at $v=1$ after cancellation}
$w_k(\eta/\lambda)$, define
\[
\defin{(M_r(\mathbf u))_{\tau,\lambda}}
\coloneqq
\sum_{\eta}
w_k(\eta/\lambda)\,u_{\xi'}.
\]
Hence $M_r(\mathbf u)$ is the weighted adjacency matrix of the
directed graph of subsection~\ref{subsection:weights}. The next lemma identifies $M_r(\uvec)$ with the specialization of $M_r(v;\uvec)$ at $v=1$.

{\Lilan
We interpret specialization at $v=1$ algebraically. Let
\[
 \mathcal O_1=\QQ(\zeta)[v]_{(v-1)}
 =\left\{\frac{a(v)}{b(v)}:
 a,b\in\QQ(\zeta)[v],\ b(1)\ne0\right\}.
\]
Evaluation defines a ring homomorphism
\[
 \operatorname{ev}_1:\mathcal O_1\longrightarrow\QQ(\zeta),
 \qquad \frac{a(v)}{b(v)}\longmapsto\frac{a(1)}{b(1)},
\]
which we apply coefficientwise to polynomials and entrywise to matrices.

}
\begin{lemma}\label{lem:limit}
Every coefficient of $M_r(v;\uvec)$ is regular at $v=1$, and
\[
 M_r(1;\uvec)=M_r(\uvec).
\]
In particular, $M_r(\uvec)$ has entries in $\QQ_{\ge0}[u_1,\ldots,u_N]$.
\end{lemma}
\begin{proof}
Fix an edge $\eta=\lambda+\delta$. By~\eqref{eq:rowpairs} and~\eqref{eq:weight-motivation}, $\psi'_{\eta/\lambda}(\zeta,\zeta v)$ is a finite product of quotients
\[
 \frac{1-\zeta^av^b}{1-\zeta^av^m},\qquad b\in\{m-1,m+1\},\quad m\ge1.
\]
If $k\nmid a$, then $\zeta^a\ne1$, so the numerator and denominator are both nonzero at $v=1$, and the quotient belongs to $\mathcal O_1$ with value $1$. If $k\mid a$, then
\[
 \frac{1-v^{b}}{1-v^m}
 =\frac{1+v+\cdots+v^{b-1}}{1+v+\cdots+v^{m-1}},
\]
whose denominator has value $m\ne0$ at $v=1$; hence the quotient belongs to $\mathcal O_1$, with value $b/m$. Thus $\psi'_{\eta/\lambda}(\zeta,\zeta v)$ is regular at $v=1$, and comparing with~\eqref{eq:beta} and~\eqref{eq:weight},
\[
 \operatorname{ev}_1\bigl(\psi'_{\eta/\lambda}(\zeta,\zeta v)\bigr)=w_k(\eta/\lambda).
\]
Each matrix entry is a finite sum of such coefficients times fixed monomials in the $u_i$. Hence every matrix entry is regular at $v=1$, and evaluating term by term gives $M_r(1;\uvec)=M_r(\uvec)$. By Lemma~\ref{lem:weights}, all coefficients of $M_r(\uvec)$ are nonnegative rational numbers.
\end{proof}

\begin{remark}\label{rem:order}
The order of specialization is important: first $q=\zeta$ with $t$ indeterminate, then $t=\zeta v$ and $v=1$. In the row-pair form~\eqref{eq:rowpairs} every denominator carries a positive power of $t$, so no factor degenerates in the first step.
{\Lilan
We evaluate the transition coefficients at $v=1$; no specialization
of the individual polynomials $P_\lambda[X_N;\zeta,\zeta v]$ at
$v=1$ is required.

}
\end{remark}

\subsection{The evaluation matrix}
Take $x_1,\ldots,x_N$ to be algebraically independent over $\QQ(\zeta,v)$, and put
\[
 L=\QQ(\zeta,v,x_1,\ldots,x_N),\qquad
 \mathcal I=\{0,\ldots,k-1\}^N.
\]
For $a\in\mathcal I$, define
\begin{equation}\label{eq:points}
 \defin{X^{(a)}}\coloneqq(\zeta^{a_1}x_1,\ldots,\zeta^{a_N}x_N).
\end{equation}
There are $k^N$ such points, and
\begin{equation}\label{eq:sameE}
 E_i(X^{(a)})=E_i(X_N),\qquad(1\le i\le N).
\end{equation}
Consider the square matrix
\begin{equation}\label{eq:U}
 U_{a,\lambda}=P_\lambda[X^{(a)};\zeta,\zeta v],
 \qquad a\in\mathcal I,\quad\lambda\in\Rset.
\end{equation}

\begin{lemma}\label{lem:invertible}
The matrix $U$ is invertible over $L$.
\end{lemma}
{\Lilan
\begin{proof}
Let $m=k^N$, $A=M_1(v;E(X_N))$, and
\[
s_a=\elementaryE_1[X^{(a)}]=\sum_{i=1}^N\zeta^{a_i}x_i,
 \qquad D=\operatorname{diag}(s_a)_{a\in\mathcal I}.
\]
The elements $s_a$ are pairwise distinct: if $a\ne b$, then
$\zeta^{a_i}-\zeta^{b_i}\ne0$ for some $i$, and algebraic
independence of the $x_i$ implies $s_a-s_b\ne0$.

For each $a\in\mathcal I$ and $\lambda\in\Rset$, evaluating
\eqref{eq:multiplication} with $r=1$ at $X^{(a)}$ gives
\[
 e_1[X^{(a)}]P_\lambda[X^{(a)};\zeta,\zeta v]
 =
 \sum_{\tau\in\Rset}
 \bigl(M_1(v;E(X^{(a)}))\bigr)_{\tau,\lambda}
 P_\tau[X^{(a)};\zeta,\zeta v].
\]
By \eqref{eq:sameE}, we have
$E_i(X^{(a)})=E_i(X_N)$ for every $i$. Consequently, $M_1(v;E(X^{(a)}))
 =
 M_1(v;E(X_N))
 =
 A$, so the same matrix $A$ occurs at every evaluation point.
Using the definitions of $s_a$ and $U$, the evaluated
multiplication identity therefore becomes $s_a U_{a,\lambda}
=
\sum_{\tau\in\Rset}
A_{\tau,\lambda}U_{a,\tau}$.

We now identify both sides as matrix entries.
Since $D=\operatorname{diag}(s_a)_{a\in\mathcal I}$,
\[
 (DU)_{a,\lambda}
 =
 \sum_{b\in\mathcal I}D_{a,b}U_{b,\lambda}
 =
 s_aU_{a,\lambda},\qquad
 (UA)_{a,\lambda}
 =
 \sum_{\tau\in\Rset}
 U_{a,\tau}A_{\tau,\lambda}.
\]
The evaluated multiplication identity shows that these entries
are equal for every $a\in\mathcal I$ and $\lambda\in\Rset$.
Hence $DU=UA$. 
Consequently $UA^j=D^jU$ for every $j\ge0$. Define the $m\times m$
matrix $C$, with rows indexed by $\Rset$ and columns
indexed by $0,\ldots,m-1$, by
\[
 C_{\lambda,j}=(A^j)_{\lambda,\varnothing}.
\]
Since $U_{a,\varnothing}=P_\varnothing[X^{(a)};\zeta,\zeta v]=1$,
we obtain
\[
\begin{aligned}
(UC)_{a,j}
=\sum_{\lambda\in\Rset}
U_{a,\lambda}(A^j)_{\lambda,\varnothing}
=(UA^j)_{a,\varnothing}
=(D^jU)_{a,\varnothing}
 =s_a^jU_{a,\varnothing}
 =s_a^j,
\end{aligned}
\]
where we used that $D$ is diagonal and
$U_{a,\varnothing}=P_\varnothing[X^{(a)};\zeta,\zeta v]=1$.
Thus $UC$ is a Vandermonde matrix at the pairwise distinct elements
$s_a$, and is therefore invertible over $L$. Both $U$ and $C$ are
square, since $|\mathcal I|=|\Rset|=m$. Hence $U$ is invertible.
\end{proof}

}

\begin{example}\label{ex:onevariable}
For $N=1$, the restricted partitions are $(d)$, $0\le d<k$, and $P_{(d)}(x)=x^d$. Hence $U_{a,d}=\zeta^{ad}x^d$, and the primitive second root of unity is $\zeta=-1$. For $k=2$ this is
\[
 U=\begin{pmatrix}1&x\\1&-x\end{pmatrix}.
\]
The two vertices are $\varnothing$ and $(1)$, and
\[
M_1(u_1)=\begin{pmatrix}0&u_1\\1&0\end{pmatrix}.
\]
After $u_1=x^2$, its eigenvalues are $x$ and $-x$, the values of $e_1$ at the two evaluation points. The next argument is the same calculation in arbitrary dimension.
\end{example}

\subsection{The root of unity filter}
Let $D_r$ be the diagonal matrix indexed by $\mathcal I$ with
\[
 (D_r)_{a,a}=\elementaryE_r[X^{(a)}].
\]
We derive the matrix identity carefully. Fix $a\in\mathcal I$ and $\lambda\in\Rset$. Evaluating~\eqref{eq:multiplication} at $X^{(a)}$ gives
\[
\elementaryE_r[X^{(a)}]P_\lambda[X^{(a)}]
=
\sum_{\tau\in\Rset}
(M_r(v;E(X_N)))_{\tau,\lambda}
P_\tau[X^{(a)}].
\]
Notice that the $(a,\lambda)$ entry of $D_rU$ is
\[
(D_rU)_{a,\lambda}=\sum\limits_{b} (D_r)_{a,b}U_{b,\lambda}=(D_r)_{a,a}U_{a,\lambda}=\elementaryE_r[X^{(a)}]P_\lambda[X^{(a)}],
\]
and the $(a,\lambda)$ entry of $UM_r(v;E(X_N))$ is
\[
(UM_r)_{a,\lambda}=\sum\limits_{\tau} U_{a,\tau}(M_r)_{\tau,\lambda}=\sum_{\tau} P_\tau[X^{(a)}] (M_r)_{\tau,\lambda}.
\]
Hence, simultaneously for all $a$ and $\lambda$,
\begin{equation}\label{eq:DU}
 D_rU=UM_r(v;E(X_N)).
\end{equation}
By Lemma~\ref{lem:invertible}, $U$ is invertible, so multiplying~\eqref{eq:DU} by $U^{-1}$ on the left gives
\begin{equation}\label{eq:similarity}
 M_r(v;E(X_N))=U^{-1}D_rU.
\end{equation}

Now let $\mu=(\mu_1,\ldots,\mu_s)$. Using~\eqref{eq:similarity} for every factor, we have
\begin{align*}
 &M_{\mu_s}(v;E(X_N))\cdots M_{\mu_1}(v;E(X_N))\\
 &\qquad=U^{-1}D_{\mu_s}UU^{-1}D_{\mu_{s-1}}U\cdots U^{-1}D_{\mu_1}U\\
 &\qquad=U^{-1}(D_{\mu_s}\cdots D_{\mu_1})U.
\end{align*}
Taking the ordinary matrix trace and using invariance of trace under similarity gives
\begin{align}
 \tr\bigl(M_{\mu_s}(v;E(X_N))\cdots M_{\mu_1}(v;E(X_N))\bigr)
  =\tr(D_{\mu_s}\cdots D_{\mu_1})=\sum_{a\in\mathcal I}\prod_{j=1}^s \elementaryE_{\mu_j}[X^{(a)}]=\sum_{a\in\mathcal I}\elementaryE_\mu[X^{(a)}].
 \label{eq:traceevaluation}
\end{align}

It remains to identify the last finite sum. Consider first a monomial
\[
 x^\alpha=x_1^{\alpha_1}\cdots x_N^{\alpha_N}.
\]
At the point $X^{(a)}$ it takes the value
\[
 (X^{(a)})^\alpha
 =\zeta^{a_1\alpha_1+\cdots+a_N\alpha_N}x^\alpha.
\]
Summing independently over $a_1,\ldots,a_N$ gives
\begin{align*}
 \sum_{a\in\mathcal I}(X^{(a)})^\alpha
 &=x^\alpha\prod_{i=1}^N
 \left(\sum_{a_i=0}^{k-1}(\zeta^{\alpha_i})^{a_i}\right).
\end{align*}
For a fixed $i$, the inner sum is a finite geometric series. If $k\mid\alpha_i$, then $\zeta^{\alpha_i}=1$ and the sum equals $k$. If $k\nmid\alpha_i$, then $\zeta^{\alpha_i}\ne1$ and
\[
 \sum_{a_i=0}^{k-1}(\zeta^{\alpha_i})^{a_i}
 =\frac{1-(\zeta^{\alpha_i})^k}{1-\zeta^{\alpha_i}}=0.
\]
Therefore
\begin{equation}\label{eq:monomial-filter}
 \sum_{a\in\mathcal I}(X^{(a)})^\alpha
 =\begin{cases}
    k^Nx^\alpha,&k\mid\alpha_i\text{ for every }i,\\
    0,&\text{otherwise}.
   \end{cases}
\end{equation}
By linearity, the finite sum over $a$ is exactly $k^N$ times the projection $\Phi_k$. Applying Lemma~\ref{lem:projection} to $e_\mu$ yields
\begin{equation}\label{eq:filter}
 \sum_{a\in\mathcal I}\elementaryE_\mu[X^{(a)}]
 =k^N(\rowMap_k \elementaryE_\mu)[\powersum_k](X_N).
\end{equation}
This explains the factor $k^N$ in the walk formula.

\begin{proposition}\label{prop:trace}
For every partition $\mu$, the positive transition matrices satisfy
\begin{equation}\label{eq:traceidentity}
 \tr\bigl(M_{\mu_s}(E(X_N))\cdots M_{\mu_1}(E(X_N))\bigr)
     =k^N(\rowMap_k \elementaryE_\mu)[\powersum_k](X_N).
\end{equation}
\end{proposition}
\begin{proof}
Equations~\eqref{eq:traceevaluation} and~\eqref{eq:filter} give the identity
\begin{equation}\label{eq:trace-v}
 \tr\bigl(M_{\mu_s}(v;E(X_N))\cdots M_{\mu_1}(v;E(X_N))\bigr)
 =k^N(\rowMap_k\elementaryE_\mu)[\powersum_k](X_N)
\end{equation}
over the rational-function field containing $v$. The right side contains no $v$.

By Lemma~\ref{lem:limit}, every entry of every matrix $M_r(v;E(X_N))$ is regular at $v=1$. A finite product of matrices with entries regular at $v=1$ again has entries regular there, and the trace is a finite sum of diagonal entries. Hence the entire left side of~\eqref{eq:trace-v} can be evaluated at $v=1$. After evaluation, each factor $M_r(v;E(X_N))$ becomes $M_r(E(X_N))$. Then, we obtain exactly~\eqref{eq:traceidentity}.
\end{proof}
\begin{remark}
    Notice that the evaluation matrix $U$ and its inverse were used only to prove~\eqref{eq:trace-v} for an indeterminate $v$. We never need to specialize $U$ itself at $v=1$. This distinction matters because the restricted Macdonald basis need not have a well-behaved limit there term by term.
\end{remark}
\section{The coefficient formula and $\elementaryE$-positivity}\label{sec:proof}
We can now identify the coefficients of the walk generating function.

\begin{theorem}\label{thm:walk}
Let $k\ge2$, $n\ge1$, $\mu\vdash kn$, $\nu\vdash n$, and $N\ge n$. Then
\[
 c^{(k)}_{\mu\nu}=\frac{1}{k^N}
        \sum_{W\in\mathcal W_{k,N}(\mu,\nu)}\wt(W).
\]
\end{theorem}
\begin{proof}
Write $s=\ell(\mu)$ and abbreviate $M_r=M_r(\uvec)$. We first expand the matrix product directly. For fixed initial and final vertices $\lambda^0,\lambda^s\in\Rset$, the $(\lambda^s,\lambda^0)$ entry of $M_{\mu_s}\cdots M_{\mu_1}$
is
\begin{equation}\label{eq:matrix-entry-walk}
 \sum_{\lambda^1,\ldots,\lambda^{s-1}\in\Rset}
 (M_{\mu_s})_{\lambda^s,\lambda^{s-1}}
 \cdots
 (M_{\mu_1})_{\lambda^1,\lambda^0}.
\end{equation}
Each matrix entry is itself a sum over all edges of the indicated size between its two vertices. Expanding all these sums, a term of~\eqref{eq:matrix-entry-walk} is therefore exactly a choice of one edge of size $\mu_j$ from $\lambda^{j-1}$ to $\lambda^j$ for every $j$. In other words, it is a walk with successive step sizes $\mu_1,\ldots,\mu_s$.

The product of the scalar parts of the chosen edge entries is the product of the edge weights, namely $\wt(W)$. The product of the label monomials is
\[
 \prod_{j=1}^s u_{(\xi^j)'}=u_{\nu(W)}.
\]
Thus the $(\lambda^s,\lambda^0)$ entry is the generating function for all such walks from $\lambda^0$ to $\lambda^s$.

The trace sums the diagonal entries, so it imposes $\lambda^s=\lambda^0$. Therefore
\begin{equation}\label{eq:tracewalks}
 \tr\bigl(M_{\mu_s}(\uvec)\cdots M_{\mu_1}(\uvec)\bigr)
 =\sum_{W\text{ closed}}\wt(W)u_{\nu(W)}.
\end{equation}
The starting vertex is part of the data because each diagonal entry is counted separately. Equation~\eqref{eq:walkdegree} shows that every label $\nu(W)$ in this sum is a partition of $n$.

Now substitute
\[
u_i=E_i(X_N)=\elementaryE_i(x_1^k,\ldots,x_N^k).
\]
Proposition~\ref{prop:trace} identifies the left side of~\eqref{eq:tracewalks} with
$ k^N(\rowMap_k\elementaryE_\mu)[\powersum_k](X_N)$.
That is,
\begin{equation}\label{eq:trace-e-expansion}
 k^N\sum_{\rho\vdash n}c^{(k)}_{\mu\rho}E_\rho(X_N).
\end{equation}
On the right side of~\eqref{eq:tracewalks}, the same substitution sends $u_{\nu(W)}$ to $E_{\nu(W)}(X_N)$. Grouping closed walks according to their label therefore gives
\begin{equation}\label{eq:trace-walk-grouped}
 \sum_{\rho\vdash n}
 \left(\sum_{W\in\mathcal W_{k,N}(\mu,\rho)}\wt(W)\right)
 E_\rho(X_N).
\end{equation}

The polynomials $E_1(X_N),\ldots,E_N(X_N)$ are algebraically independent. Since $N\ge n$, every partition $\rho\vdash n$ satisfies $\rho_1\le n\le N$, so
\[
 E_\rho=\prod_jE_{\rho_j}
\]
is a monomial in these $N$ independent generators. Distinct partitions have different multiplicity sequences $m_i(\rho)$, hence give distinct monomials $\prod_{i=1}^N E_i^{m_i(\rho)}$.

Consequently the expansion in these monomials is unique. Comparing the coefficient of $E_\nu$ in~\eqref{eq:trace-e-expansion} and~\eqref{eq:trace-walk-grouped} yields
\[
 k^Nc^{(k)}_{\mu\nu}
 =\sum_{W\in\mathcal W_{k,N}(\mu,\nu)}\wt(W),
\]
which is the required formula.
\end{proof}

\begin{theorem}\label{thm:positive}
For every $k\ge1$, $n\ge0$, and $\mu\vdash kn$, the symmetric function $\rowMap_k(\elementaryE_\mu)$ is $\elementaryE$-positive. Equivalently, $c^{(k)}_{\mu\nu}\in\ZZ_{\ge0}$ for every $\nu\vdash n$.
\end{theorem}

\begin{proof}
For $k=1$, the map $\rowMap_k$ is the identity.
For $n=0$, it sends $1$ to $1$.
In all other cases, Theorem~\ref{thm:walk} and
Lemma~\ref{lem:weights} give
$c^{(k)}_{\mu\nu}\ge0$.
These coefficients are integers, as noted in
Section~\ref{sec:intro}.
Hence $c^{(k)}_{\mu\nu}\in\ZZ_{\ge0}$, then we are done.
\end{proof}

Let $\omega$ be the involution of $\Lambda$ with $\omega(e_r)=h_r$, and define symmetric functions $A_{n,k}$ by
\begin{equation}\label{eq:Ank}
\sum_{n\ge0}\frac{A_{n,k}\,u^n}{(kn)!}=\prod_{i\ge1}\Bigl(\sum_{m\ge0}\frac{(-1)^m x_i^m u^m}{(km)!}\Bigr)^{-1}.
\end{equation}
Alexandersson and Dai proved that $\omega\rowMap_k(e_{1^{kn}})=A_{n,k}$ \cite[Theorem 6]{AlexanderssonDai2026}. The $h$-positivity of $A_{n,k}$ for general $k$ was presented by Stanley in a 2025 talk on joint work with Amdeberhan and Shareshian \cite{AmdeberhanShareshianStanley2026}; see also \cite[Theorem 5]{AlexanderssonDai2026}. 

\begin{corollary}
For all $n,k\ge1$, the symmetric function $A_{n,k}$ is a linear combination of the $h_\lambda$ with nonnegative integer coefficients.
\end{corollary}
\begin{proof}
Apply $\omega$ to Theorem~\ref{thm:positive} for $\mu=(1^{kn})$ and use $\omega\rowMap_k(e_{1^{kn}})=A_{n,k}$.
\end{proof}

\begin{corollary}\label{cor:zero}
If $\mu\vdash kn$ and $\mu_1>n$, then $\rowMap_k(e_\mu)=0$.
\end{corollary}
\begin{proof}
Write $\mu=(\mu_1,\ldots,\mu_s)$. A monomial in
\[
\elementaryE_\mu=\elementaryE_{\mu_1}\elementaryE_{\mu_2}\cdots \elementaryE_{\mu_s}
\]
is obtained by choosing, for each factor $\elementaryE_{\mu_j}$, a set of $\mu_j$ distinct variables. In particular, the first factor $e_{\mu_1}$ alone uses $\mu_1$ distinct variables, so every monomial appearing in the product has support of size at least $\mu_1$.

Suppose a monomial survives $\Phi_k$. Every exponent on its support is then a positive multiple of $k$, and therefore at least $k$. Since there are at least $\mu_1$ variables in the support, the total degree is at least $ k\mu_1$.
But $\elementaryE_\mu$ is homogeneous of degree $ |\mu|=kn$.

If $\mu_1>n$, then $k\mu_1>kn$, a contradiction. Hence $\Phi_k(\elementaryE_\mu)=0$. Lemma~\ref{lem:projection} and injectivity of the substitution $f\mapsto f[\powersum_k]$ imply $\rowMap_k(\elementaryE_\mu)=0$.
\end{proof}
The graph gives the same conclusion with $N=n$: an edge of size $\mu_1>N$ does not exist.

\begin{example}
Let $k=N=2$. The restricted partitions are $\varnothing,(1),(1,1),(2,1)$. Every edge of size $1$ is listed below; $\eta$ is included to distinguish the vertical-strip operation from the endpoint of the edge.
\begin{table}[ht]
\centering
\caption{All edges of size $1$ for $k=N=2$.}
\begin{tabular}{ccccc}
\toprule
Initial vertex $\lambda$ & Partition $\eta$
& Final vertex $\tau$ & Weight & Label\\
\midrule
$\varnothing$ & $(1)$ & $(1)$ & $1$ & $1$\\
$(1)$ & $(2)$ & $\varnothing$ & $1$ & $u_1$\\
$(1)$ & $(1,1)$ & $(1,1)$ & $2$ & $1$\\
$(1,1)$ & $(2,1)$ & $(2,1)$ & $1$ & $1$\\
$(2,1)$ & $(3,1)$ & $(1,1)$ & $1$ & $u_1$\\
$(2,1)$ & $(2,2)$ & $\varnothing$ & $2$ & $u_2$\\
\bottomrule
\end{tabular}
\label{tab:size-one-edges}
\end{table}

All edges of size $1$.
The table can be reconstructed directly from the definitions. From $\varnothing$ there is only one legal addition, giving $\eta=(1)$ and endpoint $(1)$. From $(1)$ there are two legal additions: $\eta=(2)=\varnothing+2(1)$ gives the edge to $\varnothing$ with label $u_1$, while $\eta=(1,1)$ is already restricted and gives the edge to $(1,1)$. The latter edge has weight $2$: its only contributing pair of rows is $(1,2)$, with $D=m=1$, which gives $\beta_2(1;0,1)\beta_2(2;2,1)=2$. From $(1,1)$ the only legal size-$1$ addition is $\eta=(2,1)$, giving the edge to $(2,1)$. Finally, from $(2,1)$ the two legal additions are
\[
 (3,1)=(1,1)+2(1),
 \qquad
 (2,2)=\varnothing+2(1,1).
\]
They yield, respectively, an edge to $(1,1)$ with label $u_1$ and an edge to $\varnothing$ with label $u_2$. The latter has weight $2$, again from the single pair $(1,2)$ with $D=m=1$.

In the displayed vertex order, the transition matrices are
\[
 M_1=\begin{pmatrix}
 0&u_1&0&2u_2\\1&0&0&0\\0&2&0&u_1\\0&0&1&0
 \end{pmatrix},
 \qquad
 M_2=\begin{pmatrix}
 0&0&u_2&0\\0&0&0&u_2\\1&0&0&0\\0&1&0&0
 \end{pmatrix}.
\]

All edges of size $2$.
Since $N=2$, a vertical $2$-strip must add one cell to each of the two rows. Thus each state has at most one outgoing edge of size $2$. Starting from $\varnothing$ gives $\eta=(1,1)$ and endpoint $(1,1)$; starting from $(1)$ gives $\eta=(2,1)$ and endpoint $(2,1)$. Starting from $(1,1)$ gives
\[
 \eta=(2,2)=\varnothing+2(1,1),
\]
so the endpoint is $\varnothing$ and the label is $u_2$. Starting from $(2,1)$ gives
\[
 \eta=(3,2)=(1)+2(1,1),
\]
so the endpoint is $(1)$ and again the label is $u_2$. In all four cases the weight is $1$, which gives the displayed matrix $M_2$.

Four-step closed walks and $\rowMap_2(e_1^4)$.
For step sizes $(1,1,1,1)$, we can list the closed walks explicitly. The label $(1,1)$ is obtained by using the two $u_1$-edges and no $u_2$-edge. There is one such closed walk from each initial vertex:
\[
\begin{array}{rcl}
\varnothing&\to&(1)\to\varnothing\to(1)\to\varnothing,\\
(1)&\to&\varnothing\to(1)\to\varnothing\to(1),\\
(1,1)&\to&(2,1)\to(1,1)\to(2,1)\to(1,1),\\
(2,1)&\to&(1,1)\to(2,1)\to(1,1)\to(2,1).
\end{array}
\]
In each case the two labeled edges have weight $1$, and the remaining two edges also have weight $1$, so every walk has total weight $1$. Their total contribution is therefore $4u_1^2$.

The label $(2)$ is obtained by traversing once around the four-cycle
\[
\varnothing\longrightarrow(1)\longrightarrow(1,1)       \longrightarrow(2,1)\longrightarrow\varnothing.
\]
There are four closed walks because any of the four vertices may be chosen as the initial vertex. The successive edge weights are a cyclic rotation of
\[
 1,\ 2,\ 1,\ 2,
\]
so each walk has total weight $4$. The unique labeled edge in one traversal contributes $u_2$, and the total contribution is $16u_2$. Consequently
\[
 \tr(M_1^4)=4u_1^2+16u_2,
 \qquad \rowMap_2(\elementaryE_1^4)=\elementaryE_{11}+4\elementaryE_2.
\]
 The four choices of initial vertex are counted separately; the factor $4$ in Theorem~\ref{thm:walk} comes from the root of unity filter, not from an identification of walks under rotation.

To see the normalization explicitly, substitute $u_1=E_1=\elementaryE_1[\powersum_2]$ and $u_2=E_2=\elementaryE_2[\powersum_2]$. Since $k^N=2^2=4$, the trace identity gives
\[
 4(\rowMap_2(e_1^4))[\powersum_2]=4E_1^2+16E_2.
\]
Dividing by $4$ and using injectivity of $f\mapsto f[\powersum_2]$ gives
\[
\rowMap_2(\elementaryE_1^4)=\elementaryE_{11}+4\elementaryE_2.
\]
The normalization factor $4$ comes from the root of unity average, not from quotienting the four cyclic choices of starting vertex.

The same matrices satisfy
\[
 M_1^2=u_1I+2M_2,
 \qquad
 M_2^2=u_2I.
\]
Since $M_2$ has zero diagonal, $\tr(M_2)=0$, while $I$ is a $4\times4$ identity matrix and therefore
\[
 \tr(M_2^2)=\tr(u_2I)=4u_2.
\]
For $\mu=(2,1,1)$ we have
\begin{align*}
 \tr(M_2M_1^2)
 =\tr\bigl(M_2(u_1I+2M_2)\bigr)
 =u_1\tr(M_2)+2\tr(M_2^2)
 =8u_2.
\end{align*}
After substituting $u_2=E_2$ and using the trace identity,
\[
8E_2=4(\rowMap_2(\elementaryE_2\elementaryE_1^2))[\powersum_2],\qquad
\rowMap_2(\elementaryE_2\elementaryE_1^2)=2\elementaryE_2.
\]
Likewise, $\tr(M_2^2)=4u_2$ gives
\[
 4E_2=4(\rowMap_2(\elementaryE_2^2))[\powersum_2],\qquad
\rowMap_2(\elementaryE_2^2)=\elementaryE_2.
\]
\end{example}

\textbf{AI assistance.} During the preparation of this work the authors used GPT (OpenAI) and Lean formalization in order to assist with
language editing, proof development, and the drafting of parts of the paper. After using these tools, the authors reviewed and
edited the content and take full responsibility for the publication.

\bibliographystyle{alpha}
\bibliography{theBib}

\end{document}